\documentclass[11pt,reqno]{amsart}
\usepackage{todonotes}

\title[Running title]{Sharp stability and instability of stratified steady states for the incompressible porous media equation}

\author{Seyed Abdolhamid Banihashemi, Sepehr Mohammadkhani, Huy Q. Nguyen}
\address{Department of Mathematics, University of Maryland, College Park, MD 20742}
\email[S. A. Banihashemi]{sabani@umd.edu}
\email[S. Mohammadkhani]{seperman@umd.edu}
\email[H. Q. Nguyen]{hnguye90@umd.edu}

\usepackage[margin=1in]{geometry}
\usepackage{amsmath, amsthm, amssymb, mathrsfs, stmaryrd}
\usepackage{times}
\usepackage{color}
\usepackage[colorlinks=true, pdfstartview=FitV, linkcolor=blue, citecolor=blue, urlcolor=blue]{hyperref}
\usepackage{hyperref}
\usepackage[nameinlink,noabbrev,capitalise]{cleveref}

\usepackage{hyperref}
\usepackage[nameinlink,noabbrev,capitalise]{cleveref}

\newcommand{\bq}{\begin{equation}}
\newcommand{\eq}{\end{equation}}
\newcommand{\bqa}{\begin{eqnarray*}}
\newcommand{\eqa}{\end{eqnarray*}}

\theoremstyle{plain}
\newtheorem{theo}{Theorem}[section]
\newtheorem{prop}[theo]{Proposition}
\newtheorem{lemm}[theo]{Lemma}
\newtheorem{coro}[theo]{Corollary}

\newtheorem{defi}[theo]{Definition}
\theoremstyle{definition}
\newtheorem{rema}[theo]{Remark}

\newtheorem{exam}[theo]{Example}
\DeclareMathOperator{\RE}{Re}

\DeclareMathOperator{\IM}{Im}

\DeclareSymbolFont{pletters}{OT1}{cmr}{m}{sl}
\DeclareMathSymbol{s}{\mathalpha}{pletters}{`s}

\def\tt{\theta}
\def\eps{\varepsilon}
\def\na{\nabla}

\def\mez{\frac{1}{2}}
\def\tdm{\frac{3}{2}}

\def\Rr{\mathbb{R}}
\def\T{\mathbb{T}}
\def\Nn{\mathbb{N}}
\def\Zz{\mathbb{Z}}
\def\Cc{\mathbb{C}}

\def\cL{\mathcal{L}}
\def\ld{\lambda}

\def\p{\partial}
\def\na{\nabla}
\def\wc{\rightharpoonup}

\def\om{\omega}
\def\ol{\overline}

\def\a{\alpha}
\def\b{\beta}
\def\vp{\varphi}
\def\dv{\text{div}}

\numberwithin{equation}{section}

\newcommand{\pay}{\partial_1}
\newcommand{\pad}{\partial_2}

\def\proj{\mathbb{P}}
\def\d{\delta}

\newcommand{\ddt}{\frac{d}{dt}\,}

\newcommand{\pnorm}[2]
    {
        \Vert #1 \Vert _{L^{#2}}
    }

\newcommand{\tnorm}[1]
    {
        \Vert #1 \Vert _{L^2}
    }
\newcommand{\hnorm}[2]
    {
        \Vert #1 \Vert _{H^{#2}}
    }

\begin{document}
\newcommand{\Huy}[1]{{\color{orange} \textbf{H:} #1}}
\begin{abstract}
We study the stability and instability of stratified steady states
$\rho_s=\rho_s(y)$ for the two-dimensional incompressible porous media
equation on $\mathbb{T}\times(-1,1)$ and
$\mathbb{T}\times\mathbb{R}$. On the periodic channel, uniformly
decreasing steady states satisfying natural boundary-compatibility
conditions are nonlinearly stable under small $H^m$
perturbations, for every integer $m>2$. The solutions converge in
$L^2$ to the measure-preserving stratification of the initial density
at the rate $t^{-m/2}$. Conversely, every steady state with
$\sup\rho_s'>0$ is nonlinearly unstable in $H^m$. On the infinite cylinder, we
prove nonlinear instability under the additional gap condition
\[
\sup_{\mathbb{R}}\rho_s' >\limsup_{|y|\to\infty}\rho_s'(y).
\]
In both settings, the instability is generated by positive eigenvalues
of the linearized operator converging to $\sup\rho_s'$, which equals
both its spectral bound and semigroup growth bound.
\end{abstract}


\noindent

\maketitle

\section{Introduction} 
The incompressible porous media (IPM) equation describes the evolution of a density transported by a fluid velocity that is itself determined by the density through Darcy's law:
\begin{subequations}\label{sys:IPM}
    \begin{align}
                &\partial_t \rho + u\cdot \nabla \rho = 0 \quad\text{in}~\Omega, \label{eq:mass}\\
       &u + \nabla p = -(0,\rho),\quad  \dv  u = 0 \quad\text{in}~ \Omega, \label{eq:darcy}\\
        & u\cdot n = 0  \quad\text{on}~ \partial\Omega, \label{eq:uboundary}
            \end{align}
\end{subequations}
where $\Omega\subset \Rr^2$ is a smooth domain and $n$ is the outward unit normal to $\p\Omega$.  Here $\rho(x, y, t)$ is the density, $u(x, y, t)$ is the fluid velocity, and $p(x, y, t)$ is the fluid pressure. IPM is particularly important as a model for buoyancy-driven flow in porous media. \cite{Bear}. 

Mathematical analysis of  IPM has been developed  for the  following domains: $\Rr^2$, $\T^2$, $\T\times \Rr$, and $\T\times (-1, 1)$. The velocity is related to the density via the  Biot-Savart law 
\bq\label{u in terms of rho}
    u= \nabla^\perp \Delta_D^{-1}\partial_1 \rho,\quad \na^\perp=(\p_2, -\p_1)\equiv (\p_y, -\p_x),
\eq
where $\Delta_D^{-1}$ denotes the inverse of the Laplace operator with the Dirichlet boundary condition when $\Omega=\T\times (-1, 1)$. See Lemma \ref{lemm:BS} for a rigorous derivation of \eqref{u in terms of rho}. IPM is thus an active scalar equation, where the velocity is a zeroth-order operator of  the scalar. This is similar to SQG, whose velocity is given by  $u=\na^\perp (-\Delta_D)^{-\mez} \rho$ \cite{Res, ConIgn, ConNgu}. An important difference is that the  Biot-Savart law \eqref{u in terms of rho} of IPM is {\it anisotropic} due to gravity. Nevertheless, similarly to SQG, IPM is locally well-posed in sufficiently smooth Sobolev spaces $H^m(\Omega)$ that embed into $W^{1, \infty}(\Omega)$. It suffices to take  $m>2$ for $\Omega=\Rr^2$, $\T^2$ \cite{CGO}, and $2<m\in \Nn$ for $\Omega=\T\times (-1, 1)$. Non-uniqueness of weak solutions was proven in \cite{CFG, Sze}. Small-scale creation for IPM was investigated in \cite{KisYao}, where the authors proved  infinite-in-time growth of Sobolev norms of solutions  for certain classes of initial data which are odd in $y$, provided the (smooth) solutions exist globally.  C\'ordoba and Martínez-Zoroa \cite{CorMar} constructed compactly supported smooth forcing term and initial data that lead to finite-time blowup of the $C^1$ norm of the classical solution to the forced IPM equation. For the unforced IPM equation in certain wedge domains,  Dembski \cite{Dem} proved finite-time singularity formation for Lipschitz continuous solutions which vanish on the boundary. 

This paper is concerned with the stability and instability of steady states of IPM. Clearly, any function $\rho_s(y)$ of $y$ only is a steady state of \eqref{sys:IPM}, with $u=0$ and $p(x, y)=-\int \rho_s(y)dy$.  Conversely, we have 
\begin{lemm}\cite[Lemma 1.1]{Elgindi} \label{lemm: stat density is a function of y only}
Suppose that $\rho$ is $C^1$ steady state of \eqref{sys:IPM} and $\rho$ decays fast enough at infinity when $\Omega$ is unbounded. If the function $f(x, y)=y$ is well-defined on $\Omega$, then $\rho$ is a function of $y$ only and $u= 0$. 
\end{lemm}
\begin{proof}
Multiplying the equation  $u\cdot \nabla \rho = 0$  by $y$  and integrating  by parts, we obtain  $\int_\Omega u_2 \rho = 0$.  On the other hand, by multiplying the Darcy law \eqref{eq:darcy} by $u$ and integrating by parts, we find
    \bq\label{Darcy:kinetic}
    \int_\Omega |u|^2 = - \int_\Omega u_2 \rho.
    \eq
    It follows that $\int_{\Omega} |u|^2=0$, so that  $u\equiv 0$. Then,  \eqref{eq:darcy} implies  $\rho=\rho(y)$. 
\end{proof}
We fix a steady state  $\rho_s(y)$ and perturb it by $\eta(t, x, y)$, i.e. $\rho(x, y, t)=\rho_s(y)+\eta(x, y, t)$. Since  $\p_1 \rho=\p_1\eta$, \eqref{u in terms of rho} yields 
\bq\label{eq: u in terms of eta}
u= \nabla^\perp \Delta_D^{-1}\partial_1 \eta,
\eq
and hence the transport equation \eqref{eq:mass} gives  
\bq\label{eq: transport for eta}
    \partial_t \eta -\rho_s'(y) \partial_1 \Delta_D^{-1}\partial_1\eta+ u\cdot \nabla \eta = 0.
\eq
The linearization of \eqref{eq: transport for eta} reads
\bq\label{def:L}
    \partial_t \eta = L\eta,\quad L: = \rho_s'(y) \partial_1 \Delta_D^{-1}\partial_1.
\eq
Denoting the stream function by $\psi=\Delta_D^{-1}\p_1\eta$, an integration by parts in $x$ yields 
\begin{multline}\label{L2:L}
(L\eta, \eta)_{L^2(\Omega)}=-\int_\Omega  \rho_s'(y)  \Delta_D^{-1}\partial_1\eta  \partial_1\eta \\
=-\int_\Omega  \rho_s'(y)\psi  \Delta \psi  
=\int_\Omega \rho_s'|\na  \psi|^2+\int_\Omega \rho_s'' \psi\p_2\psi \\
=\int_\Omega \rho_s'|\na  \psi|^2-\mez \int_\Omega \rho_s''' |\psi|^2
=\int_\Omega \rho_s'| u|^2-\mez \int_\Omega \rho_s''' |\psi|^2.
\end{multline}
Consequently,  $L$ is negative if $\sup_y \rho_s'(y)<0$ and  $\sup_y\rho_s'''\ge 0$. 

The first stability results concern the linearly stratified equilibrium $\rho_s(y)=-y$, which satisfies the above conditions since $\rho_s'=-1$ and $\rho_s'''=0$. Elgindi \cite{Elgindi} proved asymptotic stability of this equilibrium  for initial perturbations $\eta_0 \in W^{4, 1}(\Rr^2)\cap H^m(\Rr^2)$ and $\eta_0\in H^m(\T^2)$, $m\ge 20$. The proofs in  \cite{Elgindi} use {\it pointwise-in-time decay} estimates for the linearized equation, which  are based upon Fourier analysis facilitated by the fact that the linearized operator $L$  has a constant coefficient when $\rho_s$ is linear.  Castro, Córdoba, and Lear \cite{CCL} subsequently established the corresponding stability result in the confined periodic channel $\Omega=\T\times(-1,1)$, where the boundary effect was controlled by considering perturbations  in $H^m_e(\Omega)$, $m\ge 10$,  the Sobolev space of functions whose vertical derivatives of even orders less than $m$ vanish on the boundary.  Bianchini, Crin-Barat, and Paicu \cite{BCBP} lowered the regularity of the initial perturbations in \cite{Elgindi} to $\eta_0\in \dot H^{1-\tau}\cap \dot H^m$, $\tau \in (0, 1)$, $m\ge 3+\tau$.  For quasi-linearly stratified states, Jo and Kim \cite{JoKim} obtained quantitative asymptotic stability and sharp decay rates when $\eta_0\in H^m(\Rr^2)$, $3<m\in \Nn$. 

Park \cite{Park} introduced a different approach  which  exploits the potential-energy structure of IPM. Since the potential energy $E(\rho(t))=\int_{\Omega} \rho(x, y, t)y$ is a Lyapunov functional and satisfies
\[
\frac{d}{dt}E(\rho(t))=-\|u(t)\|_{L^2}^2,
\]
one can obtain stability from {\it time-average decay} of the velocity without relying on pointwise decay estimates for the linearized equation. In particular, Park proved stability of the linear steady state  $\rho_s(y)=-y$ in the periodic channel for initial perturbations $\eta_0\in H^k_0(\Omega)$, $2<k\in \Nn$, and identified the measure-preserving stratification of the initial data as the asymptotic state. 

The above results concerning linear (or quasi-linear) steady states provided strong evidence for the stabilizing role of  uniformly decreasing steady states, i.e. $\sup_y \rho_s'<0$. On the other hand,  for nonlinear steady states, the energetic consideration \eqref{L2:L} for $L$ suggests the necessity of the second condition $\sup_y \rho'''\ge 0$. Remarkably, Bianchini, Jo, Park, and Wang  \cite{BJPW} extended the approach in \cite{Park} to establish stability of all steady states  $\rho'_s\in C^{m+1}(\T\times \Rr)$, $2<m\in \Rr$,  that satisfy  {\it only} the first condition $\sup_\Rr\rho_s'\le -c_0<0$. Moreover, the result in \cite{BJPW} holds for any perturbations $\eta_0 \in H^m(\T\times \Rr)$, $2<m\in \Rr$, which is expected to be the sharp Sobolev threshold  since IPM is locally well-posed in $H^{2+}$. Moreover, for  periodic channels,  \cite{KisYao} proved that {\it all} smooth steady states are nonlinearly unstable in $H^{2-\gamma}$, $\gamma>0$. 

The recent progress described above gives a rather complete picture of the stability of uniformly decreasing stratified states, including the sharp Sobolev regularity threshold. This naturally raises a complementary question: is the uniform monotonicity of the background density merely a sufficient condition for stability, or is it also the mechanism that separates stable and unstable stratifications? In this paper, we address this question in the periodic channel $\T\times (-1, 1)$ and the infinite cylinder $\T\times \Rr$.  We establish a sharp stable/unstable dichotomy in the periodic channel, subject to natural boundary-compatibility conditions and apart from the borderline case. In the infinite cylinder, we obtain the corresponding instability result under an additional separation condition on the far-field behavior.\begin{theo}[Periodic channel, informal version]\label{theo:intro:1}
Let $\Omega=\T\times (-1, 1)$ and $2<m\in \Nn$.  Consider a smooth steady state $\rho_s\in W^{\infty, \infty}((-1, 1))$. 

1) If  $\sup_y\rho_s'<0$ and $\rho''_s \in H^m_e(\Omega)$, then $\rho_s$ is  nonlinearly stable in $H^m_e(\Omega)$.

2) If $\sup_y\rho_s'>0$, then $\rho_s$ is  nonlinearly unstable in $H^m(\Omega)$.
 \end{theo}
\begin{theo}[Infinite cylinder, informal version]\label{theo:intro:2}
Let $\Omega=\T\times \Rr$ and $2<m\in \Rr$. All steady states $\rho_s$ satisfying $\rho'_s\in W^{\infty, \infty}(\Rr)$ and
\bq\label{cd:cylinder:intro}
\sup_\Rr\rho_s' >\max\Big\{0, \limsup_{|y|\to \infty}\rho_s'(y)\Big\}
\eq
are nonlinearly unstable in $H^m(\Omega)$. 
\end{theo}
We refer to Theorem \ref{thm:instability} and Theorem \ref{theo:stability} below for the precise instability and stability results, respectively. 

For the periodic channel, the stability statement in Theorem \ref{theo:intro:1} is an extension of the ones in \cite{BJPW, CCL, Park}, in the regime  of integer-order regularity,  to a domain with boundary and a large class of  uniformly decreasing steady states. Our proof of the stability result exploits the interplay between the potential-energy approach in \cite{BJPW} and the boundary-compatibility mechanism identified in \cite{CCL} for linear steady states. This interplay is enabled by requiring   normal derivatives of positive even order of the steady state $\rho_s$ to vanish there (i.e. $\rho_s''\in H^m_e$). Indeed, for the channel,  the potential-energy argument alone does not control the boundary terms arising in the top-order Sobolev estimates. We overcome this difficulty by exploiting a compatible hierarchy of vanishing even-order normal derivatives, which is preserved by the IPM evolution and yields the necessary cancellations in repeated boundary integrations. The restriction $m>2$ agrees with the Sobolev threshold for local well-posedness of IPM. We use integer-order regularity because  the compatibility conditions and repeated boundary integrations are formulated in terms of classical normal derivatives. It remains  open whether the stability result for the periodic channel holds for  all real $m>2$, as in the boundaryless cylinder \cite{BJPW}.

Regarding instability,  Kiselev and Yao~\cite[Theorem 1.5]{KisYao} showed that every smooth stratified steady state in the periodic channel is nonlinearly unstable under perturbations small in the rough topology $H^{2-\gamma}$, $\gamma>0$. Their mechanism produces infinite-time growth of Sobolev norms, provided the solution remains globally smooth. In contrast, our result identifies profiles that are nonlinearly unstable in the locally well-posed regime $H^m$, $m>2$, and yields departure from equilibrium on the logarithmic time scale $T_\delta\sim |\log\delta|$. Moreover, in view of Theorem \ref{theo:intro:1} 1), the instability condition $\sup\rho_s'>0$ is sharp, apart from the borderline case $\sup \rho_s'=0$, which is interesting and deserves  separate studies.  Our instability result also has a spectral interpretation. We construct a strictly increasing sequence of positive eigenvalues  $\ld_k$ converging to 
$\sup\rho_s'$. From this and the boundedness of the linearized operator, we deduce that the spectral bound and growth bound of the linearized evolution are both equal to $\sup\rho_s'$. In particular, the unstable growth rate is determined directly by the most unstable portion of the vertical density profile.

For the infinite cylinder $\Omega=\T\times \Rr$, we recall that \cite{BJPW} proved stability of steady states $\rho'_s\in C^{m+1}(\T\times \Rr)$, $2<m\in \Rr$,  that satisfy  $\sup_\Rr\rho_s'<0$. Our instability condition \eqref{cd:cylinder:intro} consists of  the condition $\sup \rho_s'>0$ and  
\[
\limsup_{|y|\to \infty}\rho_s'(y)<\sup_\Rr\rho_s',
\]
where the latter  requires the largest positive slope of the steady density profile to occur in a bounded region rather than being approached only at spatial infinity. It therefore separates the strongest unstable stratification from the far-field behavior. At the spectral level, this strict separation places the corresponding variational levels above the essential spectrum, allowing them to be realized as discrete eigenvalues with spatially localized eigenfunctions. We refer to the proof of Theorem \ref{theo:linins:2} and Remark \ref{rema:contspec} for further details.


\section{Instability in periodic channel and infinite cylinder}
We consider the following class of steady states:
 \bq\label{def:cS}
\mathcal{U}:=\left\{\rho_s: (-1, 1)\to \Rr:  \rho_s'\in W^{\infty, \infty}((-1, 1)), \sup_{(-1, 1)}\rho_s'>0\right\} \quad\text{when~} \Omega=\T\times (-1, 1),
\eq
 \bq\label{def:cS:2}
\mathcal{U}:=\left\{\rho_s: \Rr\to \Rr:  \rho_s'\in W^{\infty, \infty}(\Rr),  \sup_\Rr\rho_s' >\max\Big\{0, \limsup_{|y|\to \infty}\rho_s'(y)\Big\}\right\}\quad\text{when~} \Omega=\T\times \Rr.
\eq
Our main result in this section asserts that all steady states in $\mathcal{U}$ are nonlinearly unstable in Sobolev spaces. 
\begin{theo}[Nonlinear instability]\label{thm:instability}
    Let   $\rho_s \in \mathcal{U}$. The linearized operator $L$ \eqref{def:L} has an eigenvalue $\ld>0$  with an associated eigenfunction $w\in H^\infty(\Omega)$ such that the following holds.  For any integer $m>2$, there exists a constant $\nu>0$ such that for any $\delta<\frac{2\nu}{\| w\|_{L^2}}$,  \eqref{eq: transport for eta} with initial data $\eta_0=\delta w$  has a unique solution $\eta \in C([0, T_\delta], H^m(\Omega))$, $T_\delta= \frac{1}{\ld }\ln \frac{2\tt}{\delta \| w\|_{L^2}}$, which satisfies $ \| \eta(T_\delta)\|_{L^2(\Omega)}\ge \nu$. Moreover, for $\Omega=\T\times \Rr$, $m$ can be taken in $(2, \infty)$. 
\end{theo}
\begin{rema}
The assumption that $\rho_s\in W^{\infty, \infty}$ in \eqref{def:cS} and \eqref{def:cS:2} is merely for convenience. For each integer $m>2$, it suffices to assume $\rho_s\in W^{K, \infty}$ for some $K>0$ depending only on $m$ and $\rho_s'$. 
\end{rema}
For the proof of Theorem \ref{thm:instability}, we first study the linearized problem for any  steady state $\rho_s\in \mathcal{U}$. We will prove that the spectrum of the  linearized operator $L$ is  contained in $\{\RE{\ld}\le \sup\rho_s'\}$ and there is a strictly increasing sequence of positive eigenvalues converging to the spectral edge $\sup\rho_s'$. By leveraging the boundedness of $L$, we deduce that both the spectral bound and the semigroup growth bound for $L$ equal $\sup \rho_s'$. With the above sharp spectral instability and semigroup bound, the passage to nonlinear instability will be achieved by using Grenier's iterative scheme \cite{Grenier}. 

\subsection{Spectral instability} We fix an arbitrary $\rho_s\in \mathcal{U}$ and recall that the linearized operator is $L=\rho_s'(y) \partial_1 \Delta_D^{-1}\partial_1$. When $\Omega=\T\times \Rr$,  $\partial_1 \Delta_D^{-1}\partial_1$ is the Fourier multiplier
\bq\label{Fourier:L}
\widehat{ \partial_1 \Delta_D^{-1}\partial_1w}(k, \xi)=
\begin{cases} \frac{k^2}{k^2+\xi^2}\hat{w}(k, \xi)\quad\text{if~}k\ne 0,\\
0\quad\text{if~} k=0.
\end{cases}
\eq
For either domain, we have
\bq\label{L:boundedness}
L\in \mathcal{L}(H^m(\Omega),  H^m(\Omega)) \quad\forall m \in [0, \infty)
\eq
and 
\bq\label{mean:Lw}
\int_{\T}Lw(x, \cdot)dx=0\quad\forall w\in L^2(\Omega).
\eq	
We note that for $\Omega=\T\times (-1, 1)$, \eqref{L:boundedness} follows by interpolation between integers $m$. 

For  $\Omega=\T\times (-1, 1)$, since $\Delta_D^{-1}$ is well-defined on $H^{-1}(\Omega)\supset L^2(\Omega)$, we have 
\bq\label{commuteL:1}
Lw=\rho_s'\p_1^2\Delta_D^{-1}w\quad\forall w\in L^2(\T\times (-1, 1)).
\eq
On the other hand, for $\Omega=\T\times \Rr$ and $w\in L^2(\Omega)$ satisfying 
\bq\label{cd:meanx}
\int_\T w(x, \cdot)dx=0,
\eq
 $\Delta^{-1}w$ is well-defined in $H^2(\Omega)$ by  
\bq\label{Fourier:inverseLaplace}
\widehat{\Delta^{-1}w}(k, \xi)=
\begin{cases} \frac{-1}{k^2+\xi^2}\hat{w}(k, \xi)\quad\text{if~}k\ne 0,\\
0\quad\text{if~} k=0.
\end{cases}
\eq
In view of this and \eqref{Fourier:L}, we deduce 
\bq\label{commuteL:2}
Lw=\rho_s'\p_1^2\Delta_D^{-1}w\quad\forall w\in L^2(\T\times \Rr)~\text{satisfying~}\eqref{cd:meanx}.
\eq
We first establish the spectral instability by proving that $L$ has positive eigenvalues. 
\begin{theo}\label{theo: linear instability}
Consider  $\Omega=\T\times (-1, 1)$. 
$L$ has a strictly increasing sequence of positive eigenvalues $\{ \ld_k\}_{k=1}^\infty$ converging to $\sup_{(-1, 1)}\rho_s'$. Moreover, for each $k\ge 1$, $\ld_k$ has a real eigenfunction $w_k(x, y)=\cos(kx)\tilde{w}_k(y)\in H^\infty(\Omega)$. 
\end{theo}
\begin{proof}
{\bf 1.} For $\Omega=\T\times (-1, 1)$, we have  
    $L = \rho_s'(y)\partial_1^2\Delta_D^{-1}$ by \eqref{commuteL:1}. If $w$ is an eigenfunction for an eigenvalue $\lambda$ of $L$, then upon setting $u = \Delta^{-1}_D w$, we have
    \bq\label{eq: u}
        \rho_s'(y)\partial_1^2 u = \lambda \Delta u.
    \eq
 In this step, we show that for any horizontal mode $k\in \Zz\setminus\{0\}$, there exists $\lambda=\ld_k>0$ and a solution to \eqref{eq: u} in the form of $u = e^{ikx} \chi(y)$, where $\chi$ is real-valued and $\chi\vert_{y=\pm 1} = 0$.   In this form, \eqref{eq: u} is equivalent to seeking $k$ and $\chi$ satisfying
    \bq\label{eq: chi}
        -\rho_s'(y)k^2\chi(y) = \lambda (-k^2 + \partial_2^2)\chi(y).
    \eq
After  rearranging  terms in \eqref{eq: chi}, multiplying by $\chi$, and integrating  by parts, we find
    \bq\label{eq: chi integral}
        \int_{-1}^1k^2(\rho_s'(y)- \lambda)|\chi(y)|^2 - \lambda|\chi'(y)|^2 dy = 0.
    \eq
    Taking the imaginary part of the above gives
\[
        \IM(\lambda)\int_{-1}^1 k^2 |\chi(y)|^2 + |\chi'(y)|^2 dy = 0.
 \]
Since the preceding integral can only vanish when $\chi\equiv 0$,  any possible eigenvalue $\ld$ corresponding to the eigenfunction $\Delta(e^{ikx}\chi(y))$ must be real.  On the other hand,  \eqref{eq: chi integral} implies that such $\ld$ must be given by
    \bq\label{form:ld:quotient}
        \lambda = \ld_k= \frac{\int_{-1}^1 \rho_s'(y)|\chi(y)|^2 dy}{\int_{-1}^1 |\chi(y)|^2 + k^{-2}|\chi'(y)|^2 dy}=:Q_k[\chi],\quad k\in \Zz\setminus\{0\}.
    \eq
 This formula implies that $\ld$, if exists, must be smaller than $\sup \rho_s'$.  For the existence of $\ld$, \eqref{form:ld:quotient} suggests considering, for any given $k\in \Zz\setminus\{0\}$, the maximization problem 
 \bq\label{max:Q}
 \max_{\chi \in H^1_0((-1, 1))\setminus\{0\}}Q_k[\chi]. 
 \eq
By homogeneity, \eqref{max:Q} is equivalent to $\max_{\mathcal{A}_k}I[\chi]$,  where 
    \bq\label{def: energy I}
       I[\chi] := \int_{-1}^1 \rho_s'(y)|\chi(y)|^2dy,\quad  J_k[\chi]:=\int_{-1}^1 |\chi(y)|^2 +k^{-2} |\chi'(y)|^2 dy = 1,
    \eq
  and  $\mathcal{A}_k=\{\chi\in H^1_0((-1, 1)): J_k[\chi]=1\}$. 

 Since $\sup_{(-1, 1)}\rho_s'>0$, there exists an open interval $U\subset (-1, 1)$ such that $\rho_s'>0$ on $U$. Choosing a function $0\neq \chi\in C^\infty(U)$, we see that 
    \bq\label{ineq: sup I is positive}
        \sup_{\chi \in \mathcal{A}_k} I[\chi] >0.
    \eq
    Also, the constraint $J_k[\chi]=1$ yields the upper bound
    \bq
        \sup_{\chi \in \mathcal{A}_k} I[\chi]  \leq \|\rho_s'\|_{L^\infty(\T)}\| \chi\|_{L^2}^2\le \|\rho_s'\|_{L^\infty(\T)},
    \eq
    so that there exists a maximizing sequence $\chi_j\in \mathcal{A}_k$.   From the definition of $J_k$, the sequence $\{\chi_j\}$ is bounded in $H^1_0((-1, 1))$, and thus  there exists  $\chi^*\in H^1_0((-1, 1))$ such that  $\chi_j \rightharpoonup \chi^*$ in $H^1((-1, 1))$ and $\chi_j \to \chi^*$ in $L^2((-1, 1))$ along a subsequence.  The strong convergence in $L^2$ implies that $I[\chi^*] = \sup_{\chi\in \mathcal{A}_k}I[\chi]$, so it remains to prove that $\chi^*\in \mathcal{A}_k$. Since \eqref{ineq: sup I is positive} holds, $\chi^*$ cannot be zero. Moreover, by the lower semi-continuity of weak convergence, we have
    \bq\label{bound:J}
    0<J_k[\chi^*] \leq\liminf_{j\to \infty} J_k[\chi_j] = 1.
    \eq
    Setting $\Bar{\chi} = (J_k[\chi^*])^{-\mez} \chi^*$, we observe that $\Bar{\chi} \in \mathcal{A}_k$ and \eqref{bound:J} implies
    $$I[\Bar{\chi}] = J_k[\chi^*]^{-1}I[\chi^*] \geq I[\chi^*] = \sup_{\chi\in \mathcal{A}_k}I[\chi],$$
    Since the preceding inequality must be an equality, we deduce that  $J_k[\chi^*] = 1$, i.e, $\chi^* \in \mathcal{A}_k$. \\
     
    Next, setting $\lambda = I[\chi^*]>0$, we claim that $\lambda$ and $\chi^*$ satisfy \eqref{eq: chi}. For any $v\in H^1_0((-1, 1))$, we define the smooth function $j:(-1,1)^2 \to \Rr$ by
    $$j(\tau, \sigma) = J_k[(1+\sigma)\chi^* + \tau v].$$
    We compute
    $$\partial_\sigma j(\tau, \sigma) = 2\int_\T ((1+\sigma)\chi^* + \tau v)\chi^* + k^{-2}((1+\sigma){\chi^*}' + \tau v') {\chi^*}' ,$$
    and so
    $$\partial_\sigma j(0,0) = 2J[\chi^*] =2\neq0.$$
    Using the implicit function theorem and the fact that $j(0,0) = J[\chi^*] =1$, there is a smooth mapping $\phi:(-\tau_0,\tau_0)\to \Rr$ such that $\phi(0)=0$ and $j(\tau, \phi(\tau)) = 1$. 
    Consider $i:(-\tau_0,\tau_0)\to \Rr$ given by 
    $$i(\tau) = I[(1+\phi(\tau))\chi^* + \tau v].$$
    Then $i$ attains its maximum at $\tau = 0$, so that 
        \bq\label{eq: derivative of i}
 0=   \frac{d }{d\tau}i(0) = 2 \int_\T \rho_s' \chi^*(\phi'(0)\chi^* + v)= 2\phi'(0) I[\chi^*] + 2\int_\T \rho_s' \chi^* v.  \eq
    Differentiating the condition $j(\tau,\phi(\tau)) = 1$ at $\tau = 0$ yields
    \bq\label{eq: derivative of phi}
    \phi'(0) = -\frac{\partial_\tau j (0,0)}{\partial_\sigma j (0,0)} =- \frac{2\int_\T \chi^* v + k^{-2}{\chi^*}' v'}{2}=- \int_\T \chi^* v + k^{-2}{\chi^*}' v'.
    \eq
Then we plug \eqref{eq: derivative of phi} into \eqref{eq: derivative of i} and recall that we have set $\lambda = I[\chi^*]$. We obtain 
    \bq
        \int_\T\rho_s' k^2 \chi^*v = \lambda \int_\T k^2\chi^* v + {\chi^*}' v' \quad \forall v\in H^1_0((-1, 1)),
    \eq
    which is precisely  the weak formulation for \eqref{eq: chi}. Using the equation \eqref{eq: chi} and a simple induction, we deduce that $\chi^*\in H^\infty((-1, 1))$. Consequently,   $w= \Delta u=\Delta(e^{ikx}\chi^*(y))\in H^\infty(\Omega)$. This proves the existence of a positive eigenvalue $\lambda=\ld_k$ and its associated smooth eigenfunction $w=w_k$ for each $k\ne 0$; moreover, since $L$ preserves reality, $\RE w_k=\cos(kx)\tilde{w}_k(y)$ is also an eigenfunction. Equivalently, for each $k\ne 0$, there exists $\chi_k\ne 0$ such that $\ld_k=\max_{\chi \ne 0}Q_k[\chi]=Q_k[\chi_k]$. Let $1\le k_1<k_2$. Since $\| \chi_1'\|_{L^2}>0$ for $\chi_1\in H^1_0((-1, 1))\setminus\{0\}$,  we have $J_{k_1}[\chi_1]>J_{k_2}[\chi_1]$, and hence 
    \[
  Q_{k_2}[\chi_{k_1}]>Q_{k_1}[\chi_{k_1}]=\ld_{k_1}.
   \]
It follows that 
\[
\ld_{k_2}=\max_{\chi \ne 0}Q_{k_2}[\chi]\ge   Q_{k_2}[\chi_{k_1}]>\ld_{k_1}.
\]
Therefore, the sequence $\{\ld_k\}_{k=1}^\infty\subset (0, \sup_{(-1, 1)}\rho_s')$ is strictly increasing. 

    {\bf 2.} Since $\ld_k=\max_{\chi\ne 0}Q_k[\chi]$ is increasing in $k$ and bounded by $\sup_{(-1, 1)}\rho_s'$, to prove that $\ld_k\to \sup_{(-1, 1)}\rho_s'$, it suffices to show the following:
    \bq\label{limldk}
    \forall \delta>0,~\exists \chi_\delta \in C^\infty_c((-1, 1))\setminus\{0\},~\exists k_\delta\ge 1,~\forall k\ge k_\delta,~Q_{k}[\chi_\delta]>\sup_{(-1, 1)}\rho_s'-\delta.
    \eq
    Since $\sup\rho_s'>0$, there exists $y_0\in (-1, 1)$ such that $\rho_s'(y_0)>\sup\rho'_s-\delta$. By the continuity of $\rho_s'$, we have $\rho_s'>\sup\rho'_s-\delta$ on some open interval $U$ about $y_0$. Fixing any cutoff function $\chi_\delta \in C^1_c(U)$, we have 
    \[
    Q_k[\chi_\delta]= \frac{\int_U \rho_s'(y)|\chi_\delta(y)|^2 dy}{\int_U |\chi_\delta(y)|^2 + k^{-2}|\chi_\delta'(y)|^2 dy} \xrightarrow{k\to \infty} \frac{\int_U \rho_s'(y)|\chi_\delta(y)|^2 dy}{\int_U |\chi_\delta(y)|^2}>\sup\rho_s'-\delta. 
    \]
 This implies \eqref{limldk},  thereby concluding  the proof of Theorem \ref{theo: linear instability}.  
  \end{proof}
\begin{theo}\label{theo:linins:2}
Consider  $\Omega=\T\times \Rr$. There exists $k_*\in \Nn$ depending only on $\rho_s'$ such that  $L$ has a strictly increasing sequence of positive eigenvalues $\{ \ld_k\}_{k=k_*}^\infty$ converging to $\sup_\Rr\rho_s'$. Moreover, for each $k\ge k_*$, $\ld_k$ has a real eigenfunction $w_k(x, y)=\cos(kx)\tilde{w}_k(y)\in H^\infty(\Omega)$. 
\end{theo}
\begin{proof}  We will adapt the proof of Theorem \ref{theo: linear instability} and only elaborate the differences. We seek a positive eigenvalue $\ld$, with associated eigenfunction $w$.  From \eqref{mean:Lw}, we have  $\int_\T w(x, y)dx=\frac{1}{\ld}\int_\T Lw(x, y)dx=0$. Consequently, \eqref{commuteL:2} gives  $Lw=\rho_s'\p_1^2u$, where $u=\Delta^{-1}w$ is well-defined by \eqref{Fourier:inverseLaplace}.  Consequently, for any $\ld>0$,  the equation $Lw=\ld w$  is equivalent to \eqref{eq: u}. Seeking a solution $u(x, y)=e^{ikx}\chi(y)$ with  $k\ne 0$ and $\chi\in H^2(\Rr)$, we again arrive at the formula \eqref{form:ld:quotient} for $\ld$:
\bq\label{ld:form:2}
\ld=\ld_k= \frac{\int_\Rr \rho_s'(y)|\chi(y)|^2 dy}{\int_\Rr |\chi(y)|^2 + k^{-2}|\chi'(y)|^2 dy}=:Q_k[\chi].
 \eq
 We denote the numerator and denominator in \eqref{ld:form:2} by $I[\chi]$ and $J_k[\chi]$, respectively.  To prove the existence of $\ld=\ld_k$, we again solve the maximization problem $\max_{\chi\in \mathcal{A}_k}I[\chi]$, where $\mathcal{A}_k=\{\chi\in H^1(\Rr): J_k[\chi]=1\}$. Consider a maximizing sequence $\chi_j\in \mathcal{A}_k$ and extract a subsequence (still denoted by $\chi_j$) $\chi_j \rightharpoonup \chi^*$ in $H^1(\Rr)$. We now prove that  $I[\chi_j]\to I[\chi^*]$ as $j\to \infty$. 
 
Since $\sup_\Rr\rho_s'>0$ for $\rho_s\in \mathcal{U}$, the proof of \eqref{limldk} carries over to $\Omega=\T\times \Rr$, yielding 
\bq\label{approx:ldk}
 \forall \delta>0,~\exists \chi_\delta \in C^\infty_c(\Rr)\setminus\{0\},~\exists k_\delta\ge 1,~\forall k\ge k_\delta,~Q_{k}[\chi_\delta]>\sup_\Rr\rho_s'-\delta.
\eq
Setting $\tilde{\chi}_\delta=\frac{\chi_\delta}{\sqrt{J_k[\chi_\delta]}}$, we have $J_k[\tilde{\chi}_\delta]=1$, and hence 
\bq\label{approx:ldk:2}
 \forall \delta>0,~\exists k_\delta\ge 1,~\forall k\ge k_\delta,~\sup_{\mathcal{A}_k}I>\sup_\Rr\rho_s'-\delta.
\eq
Since $\sup\rho_s'> \limsup_{|y|\to \infty}\rho_s'(y)$ for $\rho_s\in \mathcal{U}$, applying \eqref{approx:ldk:2} with $\delta=\sup\rho_s'-  \limsup_{|y|\to \infty}\rho_s'(y)$,  we obtain $k_*=k_\delta\ge 1$ such that $\sup_{\mathcal{A}_k}I>\limsup_{|y|\to \infty}\rho_s'(y)$ for all $k\ge k_*$. In the remainder of this proof, we fix $k\ge k_*$ and a constant $\gamma$ satisfying  
\[
\max\Big\{0, \limsup_{|y|\to \infty}\rho_s'(y)\Big\} <\gamma<  \sup_{\mathcal{A}_k}I.
\]
 There exists $R>0$ such that $\rho_s'(y)<\gamma$ for $|y|\ge R$.  We denote $r_j=\chi_j-\chi^*$ and $d = \limsup_{j} J_k[r_j]\ge 0$. Since $r_j\wc 0$ in $H^1(\Rr)$, we have
\bq\label{bound:defect}
1 = \limsup_j J_k[\chi_j] =J_k[\chi^*]+d + 2\lim\sup_j\int_\Rr \chi^*r_j + k^{-2}(\chi^*)'r_j'  =J_k[\chi^*]+d
\eq
and 
\bq\label{bound:defect:2}
    \sup_{\mathcal{A}_k}I = I[\chi^*]+\limsup_jI[r_j].
\eq
On the other hand,  using the local strong convergence of $\chi_j \to \chi^*$ in $L^2((-R,R))$ and that $\rho_s'<\gamma$ on $\{|y|>R\}$, we deduce
\bq\label{bound:defect:20}
    \limsup_j I[r_j] =  \limsup_j \int_{|y|>R} \rho'_s |r_j|^2 \leq \gamma \limsup_j \int_{|y|>R}  |r_j|^2 \leq \gamma d.
\eq
We claim that $d<1$. Indeed, if $d=1$, then \eqref{bound:defect}  implies  $\chi^*=0$, and hence \eqref{bound:defect:2}-\eqref{bound:defect:20} imply
\[
    \sup_{\mathcal{A}_k}I=\limsup_jI[r_j]\le \gamma,
\]
a contradiction.  Consequently,  $J_k[\chi^*]= 1-d>0$, so that 
\bq\label{bound:defect:3}
\sup_{\mathcal{A}_k} I\ge I\left[\frac{\chi^*}{\sqrt{J_k[\chi^*]}}\right]=\frac{I[\chi^*]}{1-d}.
\eq

Inserting this and \eqref{bound:defect:3} into \eqref{bound:defect:2}, we obtain 
\[
\sup_{\mathcal{A}_k}I \le (1-d)\sup_{\mathcal{A}_k} I+\gamma d,
\]
that is,  $d(\sup_{\mathcal{A}_k} I-\gamma)\le 0$. But $\sup_{\mathcal{A}_k} I>\gamma$, so $d=0$. Consequently,    $\chi_j \to \chi^*$ strongly in $H^1(\Rr)$, which implies  $I[\chi_j] \to I[\chi^*]$. 

Having established this, the normalization argument following \eqref{bound:J} applies verbatim and shows that $J_k[\chi^*]=1$, i.e. $\chi^*\in \mathcal{A}_k$. Thus $\ld_k=\max_{\mathcal{A}_k}I=I[\chi^*]>0$. Arguing as in the proof of Theorem \ref{theo: linear instability}, we obtain that $\chi^*\in H^\infty(\Rr)$, $\ld_k$ is an eigenvalue with associated real eigenfunction $\RE w_k=\RE \Delta(e^{ikx} \chi^*(y))=\cos(kx)\tilde{w}_k(y)\in H^\infty(\Omega)$, and $\{\ld_k\}_{k=k_*}^\infty$ is strictly increasing. Finally, \eqref{approx:ldk:2} implies $\lim_{k\to \infty}\ld_k=\sup\rho_s'$. 
 \end{proof}
 \begin{rema}\label{rema:contspec}
For $\Omega=\T\times \Rr$, the condition $\sup_\Rr\rho_s'>0$ alone is insufficient to guarantee the existence of discrete unstable eigenvalues. The simplest example is $\rho_s(y)=ay$, $a>0$. Then  $\rho_s'\equiv a$, so that  $\limsup_{|y|\to \infty}\rho_s'=\sup_\Rr\rho_s'$. Since 
\[
\widehat{Lw}(k, \xi)=\frac{a k^2}{k^2+\xi^2}\hat{w}(k, \xi),
\]
 $\sigma(L)=[0, a]$, where $(0, a]$ is the continuous spectrum and $\ld=0$ is the only eigenvalue. 
 \end{rema} 
\subsection{Semigroup Estimates}
In this subsection, we will leverage the fact that $L$ is a bounded operator to obtain sharp semigroup estimates for $e^{tL}$. Moreover, we will show that the growth rate  of $e^{tL}$ equals  $\sup \rho_s'$ and can be approximated by eigenvalues of $L$. This will be important in the proof of nonlinear instability in  Section \ref{sec:nonlinstab}.

Let $X$ be a Banach space. We denote by  $\cL(X):=\cL(X, X)$  the set of  bounded linear maps from $X$ to $X$. It is well-known that each $A\in \cL(X)$ generates the  uniformly continuous semigroup $\bigl( e^{tA}\bigr)_{t\geq 0}$. The spectrum of  $A\in \cL(X)$ is denoted  by $\sigma(A)$. 
\begin{defi}
   Let  $A\in \cL(X)$ and let  $\sigma(A)$ denote its spectrum. We define the spectral bound 
    \bq\label{def:sA}
    s(A) = \sup \{\RE(\lambda): \ld \in \sigma(A)\}
    \eq
and the growth bound 
    \bq\label{def:omA}
    \omega_0(A) = \inf\left\{\omega \in \Rr:  \exists C_\omega >0, \|e^{tA}\|_\cL \leq C_\omega e^{t\omega} ~\text{for~all~} t\ge 0\right\}.
    \eq
\end{defi}
The following result is a consequence of the spectral mapping theorem.
\begin{prop}[Chapter IV, Corollary 3.12, \cite{EN}]\label{prop: growth bound equals spectral bound}
    For any $A\in \cL(X)$, we have $s(A) = \omega_0(A)$.
\end{prop}
We recall from \eqref{L:boundedness} that $L\in \cL(H^m(\Omega))$ for all $m\in [0, \infty)$. By virtue of  Theorems \ref{theo: linear instability},  \ref{theo:linins:2}, and  Proposition \ref{prop: growth bound equals spectral bound}, we obtain 
\bq\label{som0}
\omega_0(L)= s(L)>0,
 \eq 
 where $\omega_0(L)$ in general depends on $m\ge 0$.  
  \begin{lemm}\label{lemm:spec1}  
(i)  If $\ld$ is an eigenvalue of $L\in \cL(H^m(\Omega))$, then $\ld\in \Rr$ and $\ld<\sup\rho_s'$.

(ii) If $\RE(\ld)>\sup \rho_s'$, then $\ld I-L: H^m(\Omega)\to H^m(\Omega)$ is surjective.

 \end{lemm}
 \begin{proof}
(i) Clearly $\ld=0<\sup \rho_s'$ is an eigenvalue of $L$ since $L$ annihilates all functions of $y$ only.  Suppose that $\ld \ne 0$ is an eigenvalue of $L$, and $w$ is an associated eigenfunction. We have $u:=\Delta^{-1}_Dw$ satisfies \eqref{eq: u}. For $\Omega=\T\times \Rr$, $u$ is well-defined as argued in the proof of Theorem \ref{theo:linins:2}. Expanding
 \[
 u(x, y)=\sum_{k\in \Zz} e^{ikx}\chi_k(y),
 \]
 we find that $\chi_k$ satisfies \eqref{eq: chi} in place of $\chi$. In particular, for $k=0$ we have $\p_y^2\chi_0=0$, and hence $\chi_0\equiv 0$ since either $\chi_0\in H^1_0((-1, 1))$  or $\chi_0\in H^1(\Rr)$. Since $u\not\equiv 0$, there exists $k_0\ne 0$ such that $\chi_{k_0} \not\equiv 0$.  For any such $k_0$,  after multiplying \eqref{eq: chi} by $\ol{\chi}_k$, we obtain \eqref{form:ld:quotient}: 
 \bq\label{form:ld:quotient:k}
        \lambda = \frac{\int_U \rho_s'(y)|\chi_{k_0}(y)|^2 dy}{\int_U |\chi_{k_0}(y)|^2 + k_0^{-2}|\chi'_{k_0}(y)|^2 dy},\quad U=(-1, 1)~\text{or}~U=\Rr.
    \eq
It follows that $\ld \in \Rr$ and $\ld<\sup \rho_s'$.

(ii)  Let $\RE(\ld)> \sup \rho_s'>0$ and $f\in H^m(\Omega)$. We need to find $w\in H^m(\Omega)$ such that 
\bq\label{resol:eq}
\lambda w-Lw=f.
\eq
{\it Case 1:}  $\Omega=\T\times (-1, 1)$.  By setting $u=\Delta_D^{-1}w$, we obtain the equivalent equation $\ld \Delta u-\rho_s'\p_1^2u=f$, subject to the Dirichlet boundary condition. Since $\RE(\ld)>0$, this is again equivalent to
 \bq\label{eq:uf:sur}
\cL u:= -(1-\frac{\rho_s'}{\ld})\p_1^2u-\p_2^2u=-\frac{1}{\ld}f. 
 \eq
 Since $\rho_s=\rho_s(y)$, we have 
 \bq\label{coercive:cL}
\RE (\cL u, u)_{L^2}=\RE \int_{\Omega}(1-\frac{\rho_s'}{\ld})|\p_1u|^2+|\p_2u|^2=\int_{\Omega}\big(1-\frac{\rho_s' \RE(\ld)}{|\ld|^2}\big)|\p_1u|^2+|\p_2u|^2.
 \eq
 Since $\RE(\ld)> \sup \rho_s'>0$,  we have $\rho_s' \RE(\ld)<(\RE(\ld))^2\le |\ld|^2$, so that $1-\frac{\rho_s' \RE(\ld) }{|\ld|^2}>0$. Thus,  $\cL$ is coercive, and hence \eqref{eq:uf:sur} has a unique weak solution $u\in H^1_0(\Omega)$. Finally, $u\in H^{m+2}(\Omega)$ by the standard elliptic regularity. 

{\it Case 2:}  $\Omega=\T\times \Rr$. For $g: \Omega\to \Cc$, we decompose $g=\tilde{g}+g_0$, where $g_0(y)=\int_\T g(x, y)dx$ and $\int_\T \tilde{g}(x, y)dx=0$. Since $w_0=w_0(y)$, \eqref{Fourier:L} implies $L w_0=0$. Hence, it suffices to set $w_0=\frac{1}{\ld} f_0$ and seek a solution  $\tilde{w}$ to the problem 
\bq\label{resol:eq:2}
(\ld -L)\tilde{w}=\tilde{f},\quad \tilde{w}\in H^m_*(\Omega):=\left\{g: \in H^m(\Omega): \int_\T g(x, \cdot)dx=0\right\}.
\eq
For $\tilde{w}\in H^m_*(\Omega)$,  we have $u=\Delta^{-1}\tilde{w}\in H^m_*(\Omega)$ is well-defined by \eqref{Fourier:inverseLaplace} and  $L=\rho_s'\p_1^2u$ by \eqref{commuteL:2}. Then, $u$ satisfies \eqref{eq:uf:sur} with $f$ replaced by  $\tilde{f}\in H^m_*(\Omega)$. In view of \eqref{coercive:cL} and the Poincar\'e inequality $\| g\|_{L^2(\Omega)}\le C\| \p_1 g\|_{L^2(\Omega)}$ for $g\in H^1_*(\Omega)$, $\cL$  is coercive on $H^1_*(\Omega)$ and we obtain a unique weak solution $u\in H^1_*(\Omega)$. Elliptic regularity then yields $u\in H^{m+2}_*(\Omega)$, and hence $w\in H^{m}_*(\Omega)$ as desired. 

 \end{proof}
  \begin{prop}\label{prop:spec3}
  We have
 \bq
 \sigma(L)\subset\{\ld \in \Cc: \RE(\ld)\le \sup\rho_s'\}
 \eq
 and $\omega_0(L)=s(L)=\sup\rho_s'$.
 \end{prop}
 \begin{proof}
Let $\ld\in \sigma(L)$. Suppose for the sake of contradiction that $\RE(\ld)>\sup\rho_s'$. Then $\ld I-L$ is surjective by  Lemma \ref{lemm:spec1} (ii). On the other hand, Lemma \ref{lemm:spec1} (i) implies that $\ld$ is not an eigenvalue, so that $\ld I-L$ is injective. Thus $\ld I-L$ is bijective and its inverse is bounded by the bounded inverse theorem. This contradicts the assumption that $\ld \in \sigma(L)$.

By virtue of Theorems \ref{theo: linear instability} and \ref{theo:linins:2}, $\sup \rho_s'$ is in the closure of $\sigma(L)$ which is closed (and bounded) since $L$ is bounded. Thus $\sup\rho_s'\in \sigma(L)$ is the spectral bound $s(L)$. 
 \end{proof} 
\begin{prop}\label{prop:semigroup} 
 Let $m\in [0, \infty)$, $T>0$, $\eta_0\in H^m(\Omega)$, and $f\in L^1([0, T], H^m(\Omega))$. Suppose that  there exist $M>0$ and $\omega>\omega_0(L)$ such that 
    \bq\label{bound for f}
        \|f(t)\|_{H^m(\Omega)}\leq Me^{\omega t},\quad t\in [0, T].
    \eq
      Then, there exists   $C>0$ depending only on $(m, \rho_s', \om)$ such that  the inhomogeneous problem 
    \bq\label{pb: inhomogeneous linear evolution for eta}
        \begin{cases}\partial_t \eta - L\eta = f,\\
        \eta\vert_{t=0}=\eta_0
        \end{cases}
    \eq
  has a unique solution $\eta\in C([0, T],  H^m(\Omega))$ satisfying 
    \bq\label{semigroup estimate}
        \|\eta(t)\|_{H^m(\Omega)}\leq C\big(\|\eta(0)\|_{H^m(\Omega)}+M\big) e^{\omega t},\quad t\in [0, T]. 
    \eq
\end{prop}
\begin{proof}
    The existence of $\eta$ in $C([0, T], H^m(\Omega))$ is obvious, and $\eta$ is given by the Duhamel formula 
    \bq
        \eta(t) = e^{tL}\eta(0)+  \int_0^t e^{(t-\tau)L}f(\tau) d\tau. 
    \eq
For any $\tilde{\omega}\in (\omega_0(L), \omega)$, we have $\| e^{tL}\|_{\cL(H^m)}\le Ce^{\tilde\om t}$ by the definition of $\om_0(L)$. Combining this with \eqref{bound for f} yields
        \begin{align*}
        \|\eta(t)\|_{H^m} &\leq  \|e^{tL} \eta(0)\|_{H^m} +\int_0^t \|e^{(t-\tau)L} f(\tau)\|_{H^m} d\tau \\
        &\leq C e^{\tilde\omega t}\|\eta(0)\|_{H^m} +C M \int_0^t e^{\tilde\omega (t-\tau)}e^{\omega \tau} d\tau\\
       &  =C e^{\tilde\omega t}\|\eta(0)\|_{H^m(\Omega)}+C\frac{M}{\om-\tilde\om} (e^{\omega t}-e^{\tilde\om t}),
    \end{align*}
which implies the bound \eqref{semigroup estimate}.
\end{proof}


\subsection{Nonlinear instability}\label{sec:nonlinstab}
For  ease of notation, we define 
\bq\label{def: bilinear Q}
    Q(f,g) = \nabla^\perp \Delta_D^{-1}(\partial_1 f) \cdot \nabla g  = \na \cdot (\nabla^\perp \Delta_D^{-1}(\partial_1 f)g).
\eq
Then the $\eta$-equation \eqref{eq: transport for eta} reads
\bq\label{eq: transport in terms of Q}
    \partial_t \eta- L\eta + Q(\eta,\eta)=0.
\eq
Moreover, since $\rho_s$ is purely a function of $y$, we have $Q(\rho_s, \cdot)\equiv 0$. Let $N\in \Nn$ be an integer that will be fixed later. By virtue of Theorems \ref{theo: linear instability}-\ref{theo:linins:2}  and Proposition \ref{prop:spec3}, we can pick an eigenvalue  $\ld$ of $L$ such that $\ld>\frac{\omega_0(L)}{2}=\frac{\sup\rho_s'}{2}>0$, and there is an associated real eigenfunction $w\in H^\infty(\Omega)$. 
 
{\bf 1.} We  recursively define $\eta_j$, $j=1,...,N$ by $\eta_1 =  e^{\lambda t}w$ and 
\bq \label{eq: etaj}
\begin{cases}
            \partial_t\eta_j - L\eta_j + \sum_{k=1}^{j-1} Q(\eta_k,\eta_{j-k}) = 0,\\
            \eta_j\vert_{t=0} = 0
            \end{cases}
\eq
for  $2\leq j \leq N$. Proposition \ref{prop:semigroup} implies that $\eta_j\in C([0, \infty), H^\infty(\Omega))$ for all $j\ge 1$.  Moreover, since $Q(\eta_1, \eta_1)\sim e^{2\ld t}$ and $2\ld >\om_0(L)$, applying the estimate \eqref{semigroup estimate} inductively and using the fact that $H^{n}(\Omega)$ is an algebra for $n> 1$, we find 
\bq\label{estimate for etaj}
    \|\eta_j\|_{H^n} \leq C(j, n, \| w\|_{H^{n+j-1}})e^{j\ld  t},\quad n>1,\quad 1\leq j\leq N.
\eq
We consider an approximation of the perturbation $\eta$  of the form 
\bq
    \eta^a = \sum_{j=1}^N \delta^j \eta_j,
\eq
where $\delta\in (0, 1)$ is to be determined.  Using \eqref{eq: etaj} and the fact that $\p_t \eta_1=L\eta_1$, we deduce 
\bq\label{eq:etaa}
   \p_t \eta^a-L\eta^a+Q(\eta^a, \eta^a)=\sum_{j=1}^N\sum_{i=N-j+1}^N \delta^{j+i}Q(\eta_j,\eta_i)=: R^a.
\eq
By  \eqref{estimate for etaj}, there exists $C(N)>0$ depending only on $(N, m)$ such that 
\bq\label{est:Ra}
    \|R^a(t)\|_{H^m} \leq   C(N)\Bigl(\delta^{N+1}e^{(N+1)\ld  t} + ... + \delta^{2N}e^{2N\ld  t}\Bigr).
\eq
{\bf 2.} Next, we consider the perturbed equation  \eqref{eq: transport for eta} with the initial data $\eta_0 =\delta \eta_1(0)=\delta w$. We fix any  integer $m>2$, so that $H^m(\Omega)\subset W^{1, \infty}(\Omega)$. Since $u=\nabla^\perp \Delta_D^{-1}\partial_1 \eta$ is divergence-free and has the same Sobolev regularity as  $\eta$ and $\rho_s'\in W^{\infty, \infty}(\Omega)$, it can be proven that \eqref{eq: transport in terms of Q} has a unique local solution in $H^m(\Omega)$.

For some number $\tt\in (\delta, 1)$ to be chosen, we define $T_\delta$ by 
\bq\label{def:Tdelta}
\delta e^{\ld T_\delta }=\tt,\quad i.e.\quad T_\delta =\frac{1}{\ld }\ln \frac{\tt}{\delta}>0. 
\eq
Then,  it follows from  \eqref{estimate for etaj} that 
\bq \label{est:etaa}
\| \eta^a(t)\|_{H^m}\le C(N)\frac{ \delta e^{\ld  t}}{1-\delta e^{\ld t}}\le C(N)\frac{\tt}{1-\tt},\quad t\le T_\delta,
\eq
\bq\label{est:Ra:2}
  \|R^a(t)\|_{H^m} \le C(N)\frac{\delta^{N+1}e^{(N+1)\ld  t}}{1-\tt},\quad t\le T_\delta.
\eq
Set   $v= \eta - \eta^a$, which satisfies $v\vert_{t=0}=0$. Using \eqref{eq:etaa}, we find that $v$ obeys the equation 
\bq \label{eq: evolution for v}
   \begin{aligned} 
 \partial_t v -Lv+  Q(v,\eta^a) + Q(v+\eta^a,v)+R^a =0.
\end{aligned}
\eq
Let $T_*>0$ be the maximal time such that $v$ exists on $[0, T_*)$ and $\| v(t)\|_{H^m}\le \mez$ for all $t< T_*$. We assume for the sake of contradiction that  $T_*\le T_\delta$: 
\bq\label{bootstrap:vapp}
\| v(t)\|_{H^m}\le \mez \quad\forall t<T_*. 
\eq
Next, we perform an energy estimate for $v$. For any $\alpha\in \Nn^2$ with $|\alpha|\le m\in \Nn$, we have 
\[
    \mez \frac{d}{dt} \|\partial^\alpha v\|_{L^2}^2 = \int_{\Omega} (\p^\alpha Lv)\p^\alpha v- \int_{\Omega} \partial^\alpha Q(v +\eta^a,v) \partial^\alpha v- \int_{\Omega} \partial^\alpha  \bigl(R^a + Q(v,\eta^a)\bigr) \partial^\alpha v.
\]
 Using the divergence form of $Q$ in \eqref{def: bilinear Q}, we can integrate by parts to have
\bq
\begin{aligned}
    \int_{\Omega}  \partial^\alpha Q(v+\eta^a,v) \partial^\alpha v &= \int_{\Omega} \bigl([\partial^\alpha , \nabla^\perp \Delta_D^{-1}\partial_1 (v+\eta^a)]\cdot \nabla v \bigr) \partial^\alpha v+\int_{\Omega} \bigl(\nabla^\perp \Delta_D^{-1}\partial_1 (v+\eta^a) \cdot \nabla  \p^\alpha v\bigr) \partial^\alpha v\\
    &= \int_{\Omega} \bigl([\partial^\alpha , \nabla^\perp \Delta_D^{-1}\partial_1 (v+\eta^a)]\cdot \nabla v \bigr) \partial^\alpha v.
\end{aligned}
\eq
 To handle the commutator, we recall the standard estimate (see  \cite[Lemma 3.4]{MB}):
    \bq
        \sum_{0\leq |\alpha| \leq m} \|D^\alpha(fg) - fD^\alpha g\|_{L^2(\Omega)} \leq 
        C_m \Bigl( \| f\|_{W^{1, \infty}(\Omega)} \|g\|_{H^{m-1}(\Omega)} + \| f\|_{H^m(\Omega)} \|g\|_{L^\infty(\Omega)} \Bigr).
    \eq
  Combining this with the embedding $H^{m-1}(\Omega)\subset L^\infty(\Omega)$,  we deduce 
\[
  \begin{aligned}
   \left|\int_{\T^2}  \partial^\alpha Q(v+\eta^a,v) \partial^\alpha v\right| &\leq  \|[\partial^\alpha , \nabla^\perp \Delta^{-1}\partial_1 (v+\eta^a)]\cdot \nabla v \|_{L^2}\|\partial^\alpha v\|_{L^2}\\ 
   & \leq C_m (\| v\|_{H^m}+\| \eta^a\|_{H^m})\|v\|_{H^m}^2. 
\end{aligned}
\]
It follows that 
\[
\begin{aligned}
\mez \frac{d}{dt} \|\partial^\alpha v\|_{L^2}^2& \le C_m\| v\|_{H^m}^2+C_m (\| v\|_{H^m}+\| \eta^a\|_{H^m})\|v\|_{H^m}^2+ C_m\|R^a + Q(v,\eta^a) \|_{H^m}\| v\|_{H^m}\\
&\le C_m\| v\|_{H^m}^2+C_m (\| v\|_{H^m}+\|\eta^a\|_{H^m})\|v\|_{H^m}^2+C_m \big(\|R^a\|_{H^m} +  \|v\|_{H^m}\| \eta^a\|_{H^{m+1}}\big) \| v\|_{H^m}\\
&\le C_m (\| v\|_{H^m}+\| \eta^a\|_{H^{m+1}}+1)\|v\|_{H^m}^2+ C_m\|R^a\|_{H^m} \| v\|_{H^m}.
\end{aligned}
\]
Combining this with a simple $L^2$ estimate, we obtain 
\bq\label{Hmest:verror}
\mez \frac{d}{dt} \| v\|_{H^m}^2\le C_m (\| v\|_{H^m}+\| \eta^a\|_{H^{m+1}}+1)\|v\|_{H^m}^2+C_m \|R^a\|_{H^m} \| v\|_{H^m},\quad t<T_*,
\eq
where $C_m>0$ depends only on $(m, \rho_s')$.  When $\Omega=\T\times \Rr$, the preceding estimate holds for $m\in (2, \infty)$ by commuting the $\eta$-equation with the Fourier multiplier $J^m:=(1-\Delta)^\frac{m}{2}$ by invoking Kato-Ponc\'e's commutator estimate \cite{KatoPonce}. 

Using  the estimates  \eqref{Hmest:verror}, \eqref{est:etaa}, and \eqref{est:Ra:2}, we obtain
\bq\label{mainenergy:v}
\mez \frac{d}{dt} \| v\|_{H^m}^2\le C_m \left(\| v\|_{H^m}+C(N)\frac{\tt}{1-\tt}+1\right)\|v\|_{H^m}^2+ \frac{C(N)}{1-\tt}\delta^{N+1}e^{(N+1)\ld  t}\| v\|_{H^m},\quad t<T_*.
\eq 
At this point, we successively choose $N=N(C_m, \ld)$ and $\tt=\tt(C(N))$ such that 
\[
2C_m<(N+1)\ld,\quad C(N)\frac{\tt}{1-\tt}< \mez. 
 \]
Combining this with  \eqref{bootstrap:vapp} yields
\[
C_m \left(\| v(t)\|_{H^m}+C(N)\frac{\tt}{1-\tt}+1\right)<C_m\left(\mez +\mez +1\right)<(N+1)\ld \quad\forall t<T_*.
\]
Hence, integrating \eqref{mainenergy:v} on $[0, T_*)\subset [0, T_\delta)$ and recalling \eqref{def:Tdelta}, we obtain 
\bq\label{est:v:final}
\begin{aligned}
\forall t<T_*\le T_\delta,\quad \| v(t)\|_{H^m}&\le \int_0^t e^{2C_m(t-s)}\frac{C(N)}{1-\tt} \delta^{N+1}e^{(N+1)\ld  s} ds\\
&< \frac{C(N)}{(1-\tt)[(N+1)\ld -2C_m]} \delta^{N+1}e^{(N+1)\ld  t}\\
 &<\frac{C(N)\tt^{N+1}}{(1-\tt)[(N+1)\ld -2C_m]}.
\end{aligned}
\eq
Consequently, if we further restrict $\tt=\tt(C(N))$ such that 
\[
\frac{C(N)\tt^{N+1}}{(1-\tt)[(N+1)\ld -2C_m]}<\mez,
\]
then $\sup_{t\in [0, T_*)} \| v(t)\|_{H^m}<\mez$, contradicting the maximality of $T_*$. Therefore, given  above choice of $N$ and $\tt=\tt(C(N))$, for any  $\delta \in (0, \tt)$, if $\eta(0)=\delta w$ then 
\bq
\| v(t)\|_{H^m}\le \mez \quad\forall t\le T_\delta =\frac{1}{\ld }\ln \frac{\tt}{\delta}. 
\eq
 {\bf 3.} In view of  \eqref{estimate for etaj} and  \eqref{est:v:final},  we have
\begin{align*}
\| \eta (T_\delta)\|_{L^2}& =\| \eta^a(T_\delta)+v(T_\delta)\|_{L^2}\\
&\ge \| \delta \eta_1(T_\delta)\|_{L^2}-\sum_{j=2}^N\|\delta^j \eta_j(T_\delta)\|_{L^2}-\| v(T_\delta)\|_{L^2}\\
&\ge  \tt\|w\|_{L^2}-\frac{C(N)\tt^2}{1-\tt}-\frac{C(N)\tt^{N+1}}{(1-\tt)[(N+1)\ld -2C_m]}.
\end{align*}
By choosing $\tt$ smaller if necessary, we obtain $\| \eta (T_\delta)\|_{L^2}\ge \mez \tt \|w\|_{L^2}$. This completes the proof of Theorem \ref{thm:instability} upon setting  $\nu=\mez \tt \| w\|_{L^2}$.


\section{Stability in periodic channel} 
We  define the class of steady states
\bq
\mathcal{S}_m=\left\{\rho_s\in   H^{m+2}((-1, 1)): \sup_{(-1, 1)}\rho_s'<0,~\rho_s''\in H^{m}_e\right\}, 
\eq 
where
\begin{align}\label{def:He}
    H^k_e &:= \left\{f\in H^{k}(\Omega): \p_2^{2j} f(.,\pm1) = 0,~  0\le j\le l\right\},
\end{align}
$2l$ being the largest even number that is smaller than \(k\). Below is a simple  example of  polynomials $\rho_s(y)\in \mathcal{S}_3$. 
\begin{exam} For constants $c_-<c_+$,  we consider the  polynomial
\[
\rho_s(y)=c_++\frac{c_--c_+}{2}(y+1)+\delta (y^2-1)^5,\quad\delta\in \Rr.
\]
Clearly $\rho_s(\pm 1)=c_\mp$, $\rho_s''(\pm 1)=\rho_s^{(4)}(\pm 1)=0$,  and 
\[
\rho_s'(y)=\frac{c_--c_+}{2}+10\delta y(y^2-1)^4\le \frac{c_--c_+}{2}+10\delta 
 \]
Thus, $\rho_s\in \mathcal{S}_3$ for $|\delta|<\frac{1}{20}(c_+-c_-)$. 
\end{exam}
We recall that the perturbation $\eta(x, y, t):=\rho(x, y, t)-\rho_s(y)$ satisfies 
\begin{subequations} \label{eq:pipm}
    \begin{align}
      &  \p_t \eta  +u\cdot\nabla{\eta} = -\rho_s'(y)\,u_2\quad \text{ in } \Omega,\label{eq:pmass} \\
      &   u=-\nabla q -(0,\eta),\quad
  \nabla\cdot u=0 \quad \text{ in } \Omega, \label{eq:pdarcy}\\
       & u\cdot n = 0 \quad \text{ on } \p\Omega,\label{eq:puboundary}
    \end{align}
\end{subequations}
where $q: \Omega\to \Rr$ and $u: \Omega\to \Rr^2$. 

The following  main result of this section asserts that all steady states in $\mathcal{S}_m$, $2<m\in \Nn$, are nonlinearly stable. 
\begin{theo}[Nonlinear stability]\label{theo:stability}
        Let $2<m\in \Nn$ and  $\rho_s\in\mathcal{S}_m$, with $\sup\rho_s'<-c_0<0$. There exist positive constants $M$, $M'$,  both depending only on  $(c_0, \|\rho'_s\|_{H^{m+1}}, m)$, and $\mu=\mu(c_0, \|\rho_s''\|_{W^{1, \infty}},  m)$ such that for all $\eps<(\frac{\mu}{M})^\mez$, the following holds.   For any \(\eta_0 \in H^m_e(\Omega)\) with  \(\hnorm{\eta_0}{m} < \eps\), \eqref{eq:pipm} has a unique global  solution \(\eta \in C([0,\infty), H^m_e)\), which satisfies  
        \begin{align}\label{est:mainthm}
            \hnorm{\rho(t) - \rho_s}{m}^2 + \int_0^t \hnorm{u}{m}^2 \le M\eps^2
        \end{align}
     and 
        \begin{align} \label{decay:mainthm}
            \tnorm{\rho(t) - \rho_0^*}^2 \le \frac{M'\eps^2}{t^m}
        \end{align}
      for all $t>0$, where   \(\rho(t) = \rho_s(y) + \eta(t)\) is the total density and \(\rho_0^*\) is the measure-preserving stratification of the initial density (see section \ref{sec:measure}). 
    \end{theo}
    \begin{rema}
    The proof of Theorem \ref{theo:stability} yields the following time-average decay of the velocity: 
    \bq\label{tavedecay:u}
      \frac{2}{t}\int_{t/2}^t \tnorm{u(s)}^2\, ds \le \frac{CM\eps^2}{t^{m+1}},\quad   \frac{2}{t} \int_{t/2}^t \hnorm{u_2}{2}^2 \le \frac{\tilde{C}M\eps^2}{t^m},
      \eq
      where $C=C(\|\rho_s'\|_{L^\infty}, c_0, m)$ and $\tilde{C}=\tilde{C}(\|\rho_s''\|_{W^{1, \infty}}, c_0, m)$. Moreover, 
      $\int_0^\infty \| \na u_2\|_{L^\infty}\le K(\tilde{C}, m)\sqrt{M}\eps$. See Proposition \ref{prop:potentialalgdecay}, Proposition \ref{prop:h2decay}, and Lemma \ref{lemm:nau2integ}.  
          \end{rema}
          For the linear steady state $\rho_s(y)=-y$, Theorem \ref{theo:stability} was proven in \cite{Park}. In the boundaryless infinite cylinder $\T\times \Rr$, Theorem \ref{theo:stability} was proven in \cite{BJPW} for all real Sobolev exponent $m>2$. Uniformly decreasing steady states on $\T\times \Rr$ are bijections from $\Rr$ onto $\Rr$, so they cannot be nonnegative. On the other hand, for any $\rho_s\in \mathcal{S}_m$,  $\rho_s-\inf \rho_s+1\in \mathcal{S}_m$ is positive. To obtain Theorem \ref{theo:stability}, we adapt  the potential-energy approach developed in \cite{Park, BJPW} to the domain with boundary $\Omega=\T\times (-1, 1)$.  Control of the boundary terms arising in the top-order Sobolev estimates will be achieved  by exploiting a compatible hierarchy of vanishing even-order normal derivatives, which is preserved by the IPM evolution. This boundary-compatibility mechanism was observed and exploited  in \cite{CCL} to prove stability of the linear steady state $\rho_s=-y$ in $\T\times (-1, 1)$. We will show that for general uniformly decreasing steady states, the condition $\rho_s''\in H^m_e$ allows one to integrate the boundary-compatibility mechanism with the potential-energy approach, yielding stability in $H^m$. For a self-contained proof of Theorem \ref{theo:stability}, we proceed to recall and clarify some facts about measure-preserving stratifications of functions on $\T\times (-1, 1)$ that are uniformly decreasing in the vertical variable. 
          \subsection{Measure-preserving stratification and the potential energy }\label{sec:measure}
Let \(f\in C^1(\Omega,\Rr)\cap C(\ol{\Omega}, \Rr)\) be a uniformly decreasing function, i.e., \(\p_2 f< -c_0<0\) in \(\Omega \), and assume \(f(.,\pm1) = c_\mp\) for constants \(c_0,\, c_\pm\). Then, for each \(x \in \T\), \(f(x,.):[-1,1]\to [c_-,c_+]\) is a bijection. Thus, there exists \\ \(\phi:\T\times [c_-,c_+]\to[-1,1]\) such that
\begin{equation}\label{defi:phi}
f\bigl(x, \phi(x,s)\bigr) = s\quad\forall s\in [c_-,c_+].
\end{equation}
For each $x\in \T$, $\phi(x, \cdot)$ is  uniformly decreasing and
\begin{equation}
    \frac{1}{\pnorm{\p_2 f}{\infty}}\le -\p_2 \phi(x,s) = \frac{1}{-\p_2 f\bigl(x,\phi(x,s)\bigr)} < \frac{1}{c_0}. \label{bound:pphi}
    \end{equation}
     Let \(\bar{\phi}:[c_-,c_+]\to [-1,1]\) be the $x$-average of $\phi(x, s)$,
\begin{equation}\label{def:phibar}
\bar{\phi}(s):= \frac{1}{2\pi}\int_{\T} \phi(x,s)\,dx. 
\end{equation}
It follows from  \eqref{bound:pphi} that
\begin{equation}
    \frac{1}{\pnorm{\p_2 f}{\infty}}\le -\p_2 \bar{\phi}(s) < \frac{1}{c_0}. \label{bound:pbphi}
    \end{equation}
     We define the measure-preserving stratification of \(f\), denoted by \(f^*:[-1,1]\to[c_-,c_+]\), as 
\bq
f^*(y):= \bar{\phi}^{-1}(y).
\eq
Following \cite{BJPW}, we define the potential energy
\bq\label{def:pe}
\mathcal{E}(f) = \int_\Omega f(x,y)\,y\,dxdy - \int_\Omega f^*(y)\,y\,dxdy.
\eq
The proof of Theorem \ref{theo:stability} relies on the decay of \(\mathcal{E}(\rho)\).  To this end, we use the following proposition to derive a differential inequality for \(\mathcal{E}(f)\):
\begin{prop}\label{prop:potentialequiv}
    Let \(f,f^*,\) and \(\mathcal{E}(f)\) be as above. Then we have
    \begin{align}
 \frac{c_0}{2\| \p_2f\|_{L^\infty}^2}  \tnorm{f-f^*}^2 \le \mathcal{E}(f) &\le \frac{\| \p_2f\|_{L^\infty}}{2c^2_0}\tnorm{f-f^*}^2, \label{potentialequiv} \\
    \mathcal{E}{(f)} &\le \frac{C\pnorm{\p_2 f}{\infty}}{c_0^2}\| \p_1f\|^2_{L^2} \le \frac{C\pnorm{\p_2 f}{\infty}}{c_0^2}\hnorm{\na^{\perp}(-\Delta)^{-1} \p_1 f}{1}^2 \label{potentialup},
    \end{align}
    where $C>0$ is an absolute constant. 
\end{prop}
    \begin{proof}
        Define \(h:\T\times [c_-,c_+]\to\Rr\) as 
        \begin{equation}
        h(x,s):= \phi(x,s) - \bar{\phi}(s). \label{def:h}
        \end{equation}
        First, we derive an expression of the potential in terms of \(h\). Using the change of variables \(y = \phi(x,s)\) for a fixed \(x\), we get
        \begin{align*}
            \int_\T \int_{-1}^1 f\,y\,dydx &= -\int_\T \int_{c_-}^{c_+} f\bigl(x,\phi(x,s)\bigr)\, \phi(x,s)\, \p_2\phi(x,s)\,dsdx 
            = -\int_\T \int_{c_-}^{c_+} s \, \mez \p_2\bigl(\phi^2(x,s)\bigr)\, dsdx.          
        \end{align*}
        Integrating by parts, we have 
        \begin{align*}
            \int_\T \int_{-1}^1 f\,y\,dydx &= -\mez \int_\T c_+ \, \phi^2(x,c_+)- c_- \phi^2(x,c_-)\,dx + \mez \int_\T \int_{c_-}^{c_+} \phi^2(x,s) \,dsdx.
        \end{align*}
        Similarly, the change of variables $y=\bar{\phi}(s)$ yields 
        \begin{align*}
            \int_\T \int_{-1}^{1} f^*\,y\,dydx = 
            -\mez \int_\T c_+ \, \bar{\phi}^2(c_+)- c_- \bar{\phi}^2(c_-)\,dx + \mez \int_\T \int_{c_-}^{c_+} \bar{\phi}^2(s) \,dsdx.
        \end{align*}
        Since \(f(.,\pm) = c_\mp\), it holds that \(\phi(x,c_\pm) = \bar{\phi}(c_\pm) = \mp 1\), and hence,
        \begin{align}
            \mathcal{E}(f) = \mez \int_\T \int_{c_-}^{c_+} \phi^2(x,s) - \bar{\phi}^2(s)\,dsdx
            = \mez \int_\T \int_{c_-}^{c_+} h^2(x,s)\,dsdx = \mez \tnorm{h}^2. \label{eq:potentialh}
        \end{align}
The change of variables \(y = \phi(x,s)\) gives
    \begin{align}
\|f-f^*\|_{L^2(\Omega)}^2 = \int_\T \int_{c_-}^{c_+} |f\bigl(x, \phi(x,s)\bigr) - f^*\bigl(\phi(x,s)\bigr)|^2 \, |\p_2 \phi(x,s)| \, dsdx \label{tnormffs}.
    \end{align}
    On the other hand, since \(f\bigl(x,\phi(x,s)\bigr) =s= f^*\bigl(\bar{\phi}(s)\bigr)\), we obtain
    \begin{align}
        f^*\bigl(\phi(x,s)\bigr) - f\bigl(x,\phi(x,s)\bigr)&
        =   f^*\bigl(\phi(x,s)\bigr) -f^*\bigl(\bar{\phi}(s)\bigr) \nonumber \\ 
        &= \int^{\phi(x,s)}_{\bar{\phi}(s)} \p_y f^*(y)\,dy 
        = - \int_0^{h(x,s)} \p_y f^*\bigl(y+\phi(x,s)\bigr)\,dy. \label{eq:fmfs}
    \end{align}
   By the inverse function rule and \eqref{bound:pbphi}, we have   
       \begin{align*}
        c_0 < -\p_yf^*(y) = \frac{-1}{\p_s \bar{\phi}\bigl(\bar{\phi}^{-1}(y)\bigr)} \le \pnorm{\p_2 f}{\infty},
    \end{align*}
which together with \eqref{eq:fmfs} implies 
    \[
    c_0\, |h(x,s)|\le  |f\bigl(x,\phi(x,s)\bigr) - f^*\bigl(\phi(x,s)\bigr)| \le \pnorm{\p_2 f}{\infty} \, |h(x,s)|.
    \]
By inserting this and \eqref{bound:pbphi} into \eqref{tnormffs} and  recalling  \eqref{eq:potentialh}, we obtain \eqref{potentialequiv}.

   To prove  \eqref{potentialup}, we set  $v=(-\Delta)^{-1}\p_1f$, so that  $\| \na v\|_{H^1}\ge \| \Delta v\|_{L^2}=\| \p_1f\|_{L^2}$.  Then, in view of \eqref{eq:potentialh}  and the Poincar\'e inequality $\| h\|_{L^2}\le C\| \p_1h\|_{L^2}$ (since $\int_\T h(x, s)dx=0$), it suffices to prove $ \tnorm{\p_1 f}^2 \ge \frac{c_0^2}{\pnorm{\p_2 f}{\infty}} \tnorm{\p_1 h}^2$. To this end, we differentiate the identity  \(s = f\bigl(x,\phi(x,s)\bigr)\) in $x$ to have  
    \begin{align*}
       \p_1f\bigl(x,\phi(x,s)\bigr)=-\p_2f\bigl(x,\phi(x,s)\bigr) \p_1 \phi(x,s)=-\p_2f\bigl(x,\phi(x,s)\bigr) \p_1h(x,s).
    \end{align*}
    Then, by  the change of variables $y=\phi(x, s)$ and \eqref{bound:pbphi}, we obtain
    \begin{align*}
        \tnorm{\p_1 f}^2 &= \int _\T \int_{c_-}^{c_+} |\p_1f \bigl(x,\phi(x,s)\bigr)|^2 \, |\p_2\phi(x,s)|\, dsdx\\
        &= \int_\T\int_{c_-}^{c_+} |\p_2 f\bigl(x,\phi(x,s)\bigr)|^2\, |\p_1h(x,s)|^2\, |\p_2\phi(x,s)|\,dsdx \\ 
        & \ge \frac{c_0^2}{\pnorm{\p_2 f}{\infty}} \int_\T \int_{c_-}^{c_+} |\p_1h(x,s)|^2\,dsdx
         = \frac{c_0^2}{\pnorm{\p_2 f}{\infty}} \tnorm{\p_1 h}^2.
    \end{align*}
\end{proof}
\begin{defi}[Rearrangement]
Let \(\mathcal{L}\) denote the Lebesgue measure. Two  measurable functions \(f,g:\T\times [-1,1]\to \Rr\) are called rearrangements of each other if 
    \[
    \mathcal{L}(a \le f\le b) = \mathcal{L}{(a \le g\le b)}\quad\forall a<b.
    \]
\end{defi}
\begin{prop}[Invariance of measure-preserving stratification under rearrangement]\label{prop:invstrat}
    Let $f,g\in C^1(\Omega, \Rr)\cap C(\overline{\Omega}, \Rr)$ be uniformly decreasing functions in \(y\) that are rearrangements of each other such that \(f(\pm1) = g(\pm1)= c_{\mp}\). Then, \(f^*= g^*\).
\end{prop}
\begin{proof}
We denote the  \(\phi\)  in \eqref{defi:phi} for \(f\) (resp. $g$) by \(\phi_f\) (resp. $\phi_g$).  For any $y\in [c_-, c_+]$, we have 
   \begin{multline*}
    \mathcal{L}(\{c_- \le f \le y\}) =\int_\T\int_{-1}^1 \chi_{\{c_-\le f\le y\}}(x, z)dzdx
    =-\int_\T\int_{c_-}^{c_+} \chi_{\{c_-\le f\le y\}}(x, \phi_f(x, s))\p_2\phi_f(x, s)dsdx\\
    =-\int_\T\int_{c_-}^y\p_2\phi_f(x, s)dzdx 
    = -\int_\T \phi_f(x, y)-\phi_f(x, c_-)dx 
       = -\int_\T \phi_f(x, y)-1dx.
   \end{multline*}
 Since $ \mathcal{L}(\{c_- \le f \le y\}) = \mathcal{L}(\{c_- \le g \le y\})$, we deduce that  $\int_\T \phi_f(x, y)dx=\int_\T \phi_g(x, y)dx$ for all $y\in [c_-, c_+]$. Thus $\bar{\phi}_f = \bar{\phi}_g$, i.e. $f^*=g^*$. 
\end{proof}
\begin{coro}\label{prop:potentialdecay}
    Let \(m > 2\) and suppose that \(\rho\in C([0,T], H^m)\) is a solution of \eqref{sys:IPM} such that \(\rho(\cdot ,\pm 1, t) = c_\mp\) for all $t\in [0, T]$. There exists $\eps_0=\eps_0(c_0, m)>0$ such that if 
   \begin{equation} \label{bound:hkperturb}
   \sup_{0\le t\le T} \hnorm{\rho-\rho_s}{m} \le \eps_0,
   \end{equation}
 then  \(\mathcal{E}(\rho(t))\) is well-defined,  $\rho^*(\cdot, t)=\rho^*(\cdot, 0)$ for all $t\in [0, T]$,  and 
   \begin{equation}
       \label{eq:potentialdecay}
       \frac{d}{dt} \mathcal{E}(\rho) = - \tnorm{u}^2.
   \end{equation}
\end{coro}
\begin{proof}Since $\rho_s$ is uniformly decreasing, \eqref{bound:hkperturb} and  the Sobolev embedding \(H^m(\Omega)\hookrightarrow C^{1,\a}(\overline{\Omega})\) imply that $\rho(\cdot, t)$ is uniformly decreasing for all $t\in [0, T]$.   Since $\rho$ is transported by a divergence-free vector field, the transport formula (see Proposition 1.3, \cite{MB}) implies that $\rho(\cdot, t)$ is a rearrangement of $\rho(\cdot, 0)$ for all $t\in [0, T]$. Proposition \ref{prop:invstrat} implies that $\rho^*(\cdot, t)=\rho^*(\cdot, 0)$ for all $t\in [0, T]$. Consequently, we have
\begin{align*}
\ddt \mathcal{E}(f) &= \ddt \int_\Omega \rho(x,y)\,y\,dxdy = -\int_\Omega (u\cdot\na)\rho(x,y)\, y \, dxdy= \int_\Omega u_2\rho,
\end{align*}
where the last equality follows from an  integration by parts. Then  \eqref{eq:potentialdecay} follows from this and \eqref{Darcy:kinetic}.
\end{proof}
\subsection{A priori estimates}
First, we establish the local well-posedness of \eqref{eq:pipm} in $H^m_e$ ($2<m\in \Nn$) under suitable boundary-compatibility conditions for $\rho_s$. To this end, we follow \cite{CCL} and use the following $L^2$-orthonormal basis adapted to  $H^n_e$:
\bq
\omega_{p, q}(x, y)=\frac{1}{\sqrt{2\pi}}e^{ip  x}b_q(y),\quad p\in \Zz,~q\in \Nn,
\eq
where 
\bq
b_q(y)=\begin{cases}
\cos(qy\frac{\pi}{2})\quad\text{if~}q\in 2\Nn_0+1,\\
\sin(qy\frac{\pi}{2})\quad\text{if~}q\in 2\Nn_0,\quad \Nn_0:=\Nn\cup\{0\}.
\end{cases}
\eq
Let $\proj_k$ denote the projection from $L^2$ onto $\text{span}\{ \omega_{p, q}: |p|\le k,~0\le q\le k\}$. 

Similarly,  letting $2\l+1$ be  the largest odd integer smaller than $n\ge 2$, we define 
\[
H^n_{odd}= \left\{f\in H^{n}(\Omega): \p_2^{2j+1} f(.,\pm1) = 0,~  0\le j\le l\right\}.
\]
$L^2$ has the following orthonormal basis adapted to  $H^n_{odd}$:
\bq
\tilde \omega_{p, q}(x, y)=\frac{1}{\sqrt{2\pi}}e^{ip  x}c_q(y),\quad p\in \Zz,~q\in \Nn_0,
\eq
where 
\bq
c_q(y)=\begin{cases}
\sin(qy\frac{\pi}{2})\quad\text{if~}q\in 2\Nn_0+1,\\
\cos(qy\frac{\pi}{2})\quad\text{if~}q\in 2\Nn,\\
1/\sqrt{2}\quad\text{if~}q=0.
\end{cases}
\eq
We denote by $\tilde\proj_k$  the projection from $L^2$ onto $\text{span}\{ \tilde\omega_{p, q}: |p|\le k,~0\le q\le k\}$. 
\begin{lemm}\cite[Lemma 2.4]{CCL} For $2\le n\in \Nn$, we have
\bq\label{prop:projk}
\forall f\in H^n_e, \quad \| \proj_k f\|_{H^n}\le \| f\|_{H^n}~\text{and}~\lim_{k\to \infty}\| \proj_k f- f\|_{H^n}=0,
\eq
\bq\label{prop:projk:2}
\forall f\in H^n_{odd}, \quad \| \tilde\proj_k f\|_{H^n}\le \| f\|_{H^n}~\text{and}~\lim_{k\to \infty}\| \tilde\proj_k f- f\|_{H^n}=0,
\eq
\bq\label{comm:proj:0}
\forall f\in H^1(\Omega),\quad [\p_1, \proj_k] f=[\p_1, \tilde\proj_k] f=0,\quad \p_2\tilde\proj_k f=\proj_k\p_2f,
\eq
\bq\label{comm:proj:00}
\forall f\in H^1_e(\Omega),\quad \p_2\proj_k f= \tilde\proj_k\p_2f.
\eq
\end{lemm}
It follows from \eqref{comm:proj:0} and \eqref{comm:proj:00} that 
\bq\label{comm:proj}
\forall f\in H^n_e,\quad\text{if~} j+\ell \le n~\text{then}~
\p_1^j\p_2^\ell \proj_k f=
\begin{cases}
 \proj_k \p_1^j\p_2^\ell f\quad\text{if~} \ell~\text{is~even},\\
  \tilde\proj_k \p_1^j\p_2^\ell f\quad\text{if~} \ell~\text{is~odd}.
\end{cases}
\eq
The following theorem generalizes Theorem 4.1 in \cite{CCL} from $\rho_s=-y$ to $\rho_s'\in H^m$ satisfying $\rho_s''\in H^{m-2}_e$.
\begin{prop}\label{prop:lwp}
 Let \(2<m\in \Nn\) and assume  $\rho_s'\in H^m$  and $\rho_s''\in H^{m-2}_e$. For any $\eta_0\in H^m_e$, there exist $T>0$ and a unique solution $\eta\in C([0, T], H^m_e)$ to  \eqref{eq:pipm}. Moreover, if the maximal existence time $T^0$ is finite, then 
 \bq\label{blowup}
\int_0^{T_0}  \| \na \eta(s)\|_{L^\infty}+\| \na u(s)\|_{L^\infty} ds=\infty. 
 \eq
\end{prop}
\begin{proof}  
  For $k\in \Nn$, we consider the $k^{th}$ Galerkin approximation $\eta^{[k]}$ which solves the problem 
\bq\label{eq:Galerkin}
\begin{cases}
\p_t\eta^{[k]}+\proj_k(u^{[k]}\cdot \na \eta^{[k]})=-\proj_k(\rho_s' u^{[k]}_2),\quad u^{[k]}=\na^\perp\Delta_D^{-1}\p_1\eta^{[k]},\\
\eta^{[k]}\vert_{t=0}=\proj_k\eta_0.
\end{cases}
\eq
The solution must be of the form 
\[
\eta^{[k]}(x, t)=\sum_{|p|\le k, q\le k} a_{p, q}(t)\omega_{p, q}(x).
\]
Then \eqref{eq:Galerkin} becomes an autonomous ODE system for  $\{a_{p, q}\}_{|p|\le k, q\le k}$. The local existence of $a_{p, q}$ follows from the Cauchy-Lipschitz theorem. Since $\proj_k$ is a projection, we have 
\bq\label{eest:Ga:0}
  \mez \ddt \tnorm{\eta^{[k]}}^2 = - \int_\Omega \rho'_s(y) \,u^{[k]}_2 \, \eta^{[k]} \le \pnorm{\rho'_s}{\infty} \tnorm{u^{[k]}}\tnorm{\eta^{[k]}}\le C\pnorm{\rho'_s}{\infty} \tnorm{\eta^{[k]}}^2,\quad C=C(m).
\eq
It follows that 
\[
\tnorm{\eta^{[k]}(t)}\le \tnorm{\eta_0^{[k]}}\exp(Ct \pnorm{\rho'_s}{\infty}),
\]
and hence $a_{p, q}(t)$ exists for all $t>0$. Clearly $\eta^{[k]}\in C^1([0, \infty); H^\infty_e)$ for all $n\in \Nn$, where $H^\infty_e:=\cap_{n\in \Nn} H^n_e$.  To obtain the  uniform $H^m$ estimate for $\eta^{[k]}$, we claim that 
\bq\label{vanishing:lwp}
u^{[k]}\cdot \na \eta^{[k]},~\rho_s'u_2^{[k]}\in H^{m}_e.
\eq
Since $\eta^{[k]}\in H^\infty_e$, we have  $u^{[k]}_2=-\p_1^2\Delta^{-1}\eta^{[k]}\in H^\infty_e$, $\p_2u^{[k]}_1=\p^2_2\Delta^{-1}\p_1\eta^{[k]}\in H^\infty_e$, and hence
\[
u^{[k]}\cdot \na \eta^{[k]}=u^{[k]}_1\p_1 \eta^{[k]}+u^{[k]}_2 \p_2 \eta^{[k]}\in H^\infty_e
\]
by Leibniz's rule. Likewise, since $u^{[k]}_2\in H^\infty_e$, $\rho_s'\in H^m$, and $\rho_s'' \in H^{m-2}_e$, we have $\rho_s'u_2^{[k]}\in H^{m}_e$. Indeed, for any $j\le l$, $2l$ being the largest  even number smaller than $m$, we have 
\[
\p_2^{2j}(\rho_s'u_2^{[k]})=\sum_{r=0}^{2j}a_r\p_2^{r}\rho_s'\p_2^{2j-r}u_2^{[k]}.
\]
If $r$ is even, then $\p_2^{2j-r}u_2^{[k]}$ vanishes on $\p\Omega$ since $u^{[k]}_2\in H^\infty_e$. If $r$ is odd, then $\p_2^{r}\rho_s'=\p_2^{r-1}\rho_s''$ vanishes on $\p\Omega$ because $r-1$ is even, $r-1\le 2l-2$, and  $\rho_s''\in H^{m-2}_e$.

Since $\rho_s'u_2^{[k]}\in H^{m}_e$, we can apply \eqref{prop:projk} to have 
\bq\label{eest:Ga:1}
\| \proj_k(\rho_s'u_2^{[k]})\|_{H^m}\le \| \rho_s'u_2^{[k]}\|_{H^m}\le C_m\| \rho_s'\|_{H^m}\| u_2^{[k]}\|_{H^m}\le C_m\| \rho_s'\|_{H^m}\| \eta^{[k]}\|_{H^m}, 
\eq
where $C_m$ is independent of  $k$. For any multi-index $|\alpha|=m$, since $u^{[k]}\cdot \na \eta^{[k]}$, $\eta^{[k]}\in H^\infty_e$, the  commutation  identity  \eqref{comm:proj} implies
\[
\Big(\p^\alpha   \proj_k(u^{[k]}\cdot \na \eta^{[k]}), \p^\alpha \eta^{[k]}\Big)_{L^2}=\Big(\p^\alpha   (u^{[k]}\cdot \na \eta^{[k]}), \p^\alpha \proj_k\eta^{[k]}\Big)_{L^2}=\Big(\p^\alpha  (u^{[k]}\cdot \na \eta^{[k]}), \p^\alpha\eta^{[k]}\Big)_{L^2}.
\]
 Consequently, for $m>2$, the standard commutator estimate for $[\p^\alpha, u^{[k]}]\cdot \na \eta^{[k]}$ gives
 \begin{align}\label{eest:Ga:20}
\Big| \Big(\p^\alpha   \proj_k(u^{[k]}\cdot \na \eta^{[k]}), \p^\alpha \eta^{[k]}\Big)_{L^2}\Big| & \le C_m\big(\|\na u^{[k]}\|_{L^\infty}\|  \eta^{[k]}\|_{H^m}+\|\na \eta^{[k]}\|_{L^\infty}\|  u^{[k]}\|_{H^m}\big)\|  \eta^{[k]}\|_{H^m}\\
&\le C_m\|u^{[k]}\|_{H^m}\|  \eta^{[k]}\|_{H^m}^2\le C_m\|  \eta^{[k]}\|_{H^m}^3. \label{eest:Ga:2}
 \end{align}
Combining  \eqref{eest:Ga:1} and \eqref{eest:Ga:2} with \eqref{eest:Ga:0}, we obtain
 \bq\label{apriori:Galarkin}
\mez \frac{d}{dt}\| \eta^{[k]}\|_{H^m}^2\le C_m \|\eta^{[k]}\|_{H^m}^3+C_m\| \rho_s'\|_{H^m} \|\eta^{[k]}\|_{H^m}^2, 
\eq
where $C_m$ is independent of  $k$. This implies a uniform bound for $\|\eta^{[k]}\|_{H^m}$ on a $k$-independent interval $[0, T]$. The solution $\eta\in C([0, T]; H^m_e)$ is then obtained by taking the limit of $\eta^{[k]}$. We note in particular that \eqref{vanishing:lwp} is needed to pass $k\to \infty$ in the projections $\proj_k$ in \eqref{eq:Galerkin}. We refer to the proof of \cite[Theorem 4.1]{CCL} for further details. 

Let $T^0$ be the maximal existence time of $\eta$. From  \eqref{eest:Ga:1} and \eqref{eest:Ga:20}, we have 
\[
\mez \frac{d}{dt}\| \eta^{[k]}\|_{H^m}^2\le C_m \big(\| \na \eta\|_{L^\infty}+\| \na u\|_{L^\infty}+\| \rho_s'\|_{H^m}\big)\|\eta^{[k]}\|_{H^m}^2,
\]
and hence 
\[
\sup_{s\in [0, t]}\| \eta^{[k]}(s)\|_{H^m}\le \| \eta^{[k]}(0)\|_{H^m}\exp\left\{\int_0^t \big(\| \na \eta(s)\|_{L^\infty}+\| \na u(s)\|_{L^\infty}+\| \rho_s'\|_{H^m}\big)ds\right\},\quad t<T^0. 
\]
This implies the blowup criterion \eqref{blowup}. 
\end{proof}
For solutions $\eta\in C([0, T]; H^m_e)$ satisfying the extra regularity $\eta(t)\in H^{m+1}(\Omega)$, we shall establish in Proposition \ref{prop:aprioriestimate} an $H^m$ estimate which reveals the dissipation mechanism of the term $\rho_s'u_2$ in \eqref{eq:pmass}. To this end, we first prove the vanishing on $\p\Omega$ of certain  vertical derivatives of $u_1$ and $u_2$:
\begin{prop} \label{prop:hme}
Suppose that \(\eta\in C([0,T],H^m_e)\) is a solution of \eqref{eq:pipm}. Then we have   $u_2(\cdot, t)\in H^{m}_e$ and  $\p_2 u_1(\cdot, t)\in H^{m-2}_e$ for all \(t\in [0,T]\). If  \(\eta (\cdot, t)\in H^{m+1}(\Omega)\) for all \(t\in [0,T]\), then  $u_2(\cdot, t)\in H^{m+1}_e$ and $\p_2 u_1(\cdot, t)\in H^{m}_e$ for all $t\in [0, T]$. 
\end{prop}
    
\begin{proof}
Since $u_2(\cdot, \pm 1)=0$ and $\eta\in H^m_e$,  the second component of  \eqref{eq:pdarcy} yields
    \begin{align*}
        \p_2q(.,\pm1) = -u_2(.,\pm1)-\eta(.,\pm1) =0.
    \end{align*}
    Thus, taking \(\p_2\) of the first component of \eqref{eq:darcy} gives
    \begin{align} \label{eq:tuo}
        \p_2u_1(.,\pm1) = -\p_2\p_1q(.,\pm1) =  - \p_1\p_2q(.,\pm1) = 0.
    \end{align}
    Combining this with the incompressibility condition, we obtain
    \begin{align} \label{eq:tut}
        \p_2^{2}u_2(.,\pm1) =  -\p_1\p_2u_1(.,\pm1)= 0.
    \end{align}    
    Let \(1\le j \le l\), where $2l$ is the largest even number that is less than $m$.  We  assume for the sake of  induction that 
    \bq\label{induction}
 \p_2^{2i-1}u_1(.,\pm1)=\p_2^{2i}u_2(.,\pm1) = 0\quad\forall 1\le i < j.
    \eq 
Taking \(\p_2^{2j-2}\) of the second component of \eqref{eq:pdarcy} and using \eqref{induction}, we obtain
    \begin{align*}
          \p_2^{2j-1}q(.,\pm1) = -\p_2^{2j-2}u_2(.,\pm1)  - \p_2^{2j-2}\eta(.,\pm1) = 0.
    \end{align*}
Then we take \(\p_2^{2j-1}\) of the first component of \eqref{eq:darcy} to have
\[
        \p_2^{2j-1}u_1(.,\pm1) = -\p_1\p_2^{2j-1}q(.,\pm1) = 0.
\]
    Hence, the incompressibility implies
\[
        \p_2^{2j}u_2(.,\pm1)  = -\p_1\p_2^{2j-1}u_1(.,\pm1) = 0.
\]
    We have shown that 
    \bq\label{induction:2}
   \p_2^{2i-1}u_1(.,\pm1)=\p_2^{2i}u_2(.,\pm1) = 0\quad\forall 1\le i \le l, 
    \eq
and hence $u_2\in H^m_e$ and $\p_2u_1\in H^{m-2}_e$. Taking \(\p_2^{2l}\) of the second component of \eqref{eq:pdarcy} and using \eqref{induction:2}, we obtain
    \begin{align*}
          \p_2^{2l+1}q(.,\pm1) = -\p_2^{2l}u_2(.,\pm1)  - \p_2^{2l}\eta(.,\pm1) = 0.
    \end{align*}
So far we have only used the assumption $\eta(t)\in H^m_e$. Now we assume $\eta(t)\in H^{m+1}(\Omega)$. Then we take \(\p_2^{2l+1}\) of the first component of \eqref{eq:pdarcy} to have
\[
        \p_2^{2l+1}u_1(.,\pm1) = -\p_1\p_2^{2l+1}q(.,\pm1) = 0,
\]
thereby concluding $\p_2u_1\in H^{m}_e$.   If $m$ is odd, then the  largest even integer less than $m+1$ is $2l$, so that $u_2\in H^{m+1}_e$ in view of \eqref{induction:2}. If $m$ is even, then the  largest even integer less than $m+1$ is $m=2l+2$. In this case, the incompressibility  implies 
\[
\p_2^{2l+2}u_2(\cdot, \pm 1)=-\p_1\p_2^{2l+1}u_1(\cdot, \pm 1)=0,
\]
where the traces make sense because $u\in H^{m+1}=H^{2l+2+1}$. Therefore, $u_2\in H^{m+1}_e$ in either case. 
\end{proof}
\begin{prop}\label{prop:aprioriestimate}
    Let \(2<m\in \Nn\) and $\rho_s\in \mathcal{S}_m$, with $\sup \rho_s'<-c_0<0$. Suppose that \(\eta\in C([0,T],H_e^{m})\) is a solution to \eqref{eq:pipm} such that  \(\eta(\cdot, t)\in H^{m+1}(\Omega)\) for all $t\in [0, T]$. Then there exist positive constants $A_1=A_1(m)$ and $A_2=A_2(\|\rho'_s\|_{H^{m+1}}, m)\) such that
    \begin{align}\label{estimate:thetahm}
        \mez \ddt \hnorm{\eta}{m}^2 \le -\frac34c_0\hnorm{u}{m}^2 
        + A_1\left(\pnorm{\na u_2}{\infty} \hnorm{\eta}{m}^2 +  \hnorm{u}{m}^2\hnorm{\eta}{m}\right)+A_2\tnorm{u}^2\quad\forall t\in (0, T).
    \end{align}
\end{prop}
\begin{proof} 
By Proposition \ref{prop:hme} and the additional assumption
$\eta(\cdot,t)\in H^{m+1}(\Omega)$, we have
\[
u_2(\cdot, t)\in H^{m+1}_e,\quad  \p_2 u_1(\cdot, t)\in H^{m}_e\quad\forall t\in [0, T].
\]
To prove \eqref{estimate:thetahm}, we start with the basic $L^2$ estimate 
    \begin{align}\label{estimate:thetal2}
        \mez \ddt \tnorm{\eta}^2 = - \int_\Omega \rho'_s(y) \,u_2 \, \eta \le \pnorm{\rho'_s}{\infty} \tnorm{u}\tnorm{\eta}.
    \end{align}
Let $D^m$ be an arbitrary $x$-partial derivative of order $m$. We have
    \begin{align} \label{eq:thetak}
        \mez \ddt \tnorm{D^m \eta}^2 = - \int_\Omega D^m (u\cdot\na\eta) \, D^m \eta + \int_\Omega D^m(-\rho'_s\, u_2)\, D^m \eta =: I + II.
    \end{align}
   {\bf 1.}  Expanding \(I\) results in
    \begin{align*}
        I = \sum_{1 \le i\le m} a_{i,m} \int_\Omega D^iu_1\, D^{m-i}\pay \eta \; D^m \eta + a_{i, m} \int_\Omega D^i u_2\, D^{m-i} \pad\eta \, D^m \eta.
    \end{align*}
    By Lemmas \ref{lemm:stream} and \ref{lemm:BS}, $u$ admits  a stream function \(\psi\) vanishing on $\p\Omega$ and satisfying $\Delta \psi=\p_1 \eta$. To bound the first term in $I$, we note 
           \begin{equation} \label{ineq:pthetahk}
        \hnorm{\pay \eta}{m-i} = \hnorm{\Delta \psi}{m-i}= \hnorm{\na\cdot (\na \psi)}{m-i}  \le C_m \hnorm{u}{m-i+1}.
    \end{equation}
 For  \(1\le i\le m\), we can choose \( p\in [2, \infty]\) and $q\in [2, \infty]$ such that \(\frac{1}{p} + \frac{1}{q} = \mez\) and the embeddings \(H^m \hookrightarrow W^{i,p}, \; H^{m-1}\hookrightarrow W^{m-i,q}\) hold.
    Using H\"{o}lder's inequality followed by the above embeddings, we get
    \bq\label{Lpqest}
    \pnorm{D^i u_1}{p}\, \pnorm{D^{m-i} \pay \eta}{q}\le C\| u\|_{H^m}^2,
    \eq
    and hence 
    \begin{align*}
    \left|\int_\Omega D^i u_1 \, D^{m-i}\pay \eta \, D^m \eta \right|
    \le \pnorm{D^i u_1}{p}\, \pnorm{D^{m-i} \pay \eta}{q}\, \hnorm{D^m\eta}{2} \le \hnorm{u}{m}^2\,\hnorm{\eta}{m}.
    \end{align*}
  Next, we consider  the second term in \(I\). If \(i = 1,\) then
    \[
    \left|\int_\Omega Du_2\, D^{m-1}\pad \eta \, D^m \eta\right| \le \pnorm{\na u_2}{\infty} \, \hnorm{\eta}{m}^2.
    \]
Next, we consider \(i\ge 2\) and write  \(D^m = \pay^a \pad^b\), $a+b=m$. If \(a > 0\), then
\[
\tnorm{D^m \eta} =\tnorm{\p_1^{a-1}\p_2^b \p_1\eta}\le C_m\|\p_1\eta \|_{H^{m-1}}\le  C_m \hnorm{u}{m},
\]
 and the same choice of \(p,q\) above yields
    \[
   \left| \int_\Omega D^i u_2\, D^{m-i}\p_2\eta \, D^m \eta \right|
    \le C_m \hnorm{u}{m}^2\hnorm{\eta}{m}.
    \]
    Now, consider the case \(D^m = \pad^m\). Writing $u_2=-\p_1\psi$, we can integrate by parts in \(\pay\) to  obtain 
    \begin{align}
        \int_\Omega  D^i u_2\; D^{m-i}\pad \eta \; D^m \eta 
        &= \underbrace{\int_\Omega \p_2^i \psi \; \p_2^{m-i+1} \pay \eta \; \p_2^m \eta}_{K_j} -  \underbrace{\int _\Omega \pad^i \psi \; \pad^{m-i+1} \eta \; \pad^m\pay \eta}_{J_i} \label{term:transport22}.
    \end{align}
By an  $L^p$-$L^q$ estimate  analogous to \eqref{Lpqest} (with $i-1$ in place of $i$), we have 
\[
|K_j|=\left|\int_\Omega \p_2^{i-1} u_1 \; \p_2^{m-(i-1)}\pay \eta \; \p_2^m \eta\right|\le  \hnorm{u}{m}^2\, \hnorm{\eta}{m}. 
\]
    To bound $J_i$, we integrate by parts in \(\pad\):
    \begin{align*}
        J_i &=- \int_\Omega \pad^{i+1}\psi\; \pad^{m-i+1}\eta \; \pad^{m-1}\pay \eta  - \int_\Omega \pad^i \psi \; \pad^{m-i+2}\eta \; \pad^{m-1}\pay \eta  \pm \int_{y=\pm1} \pad ^i \psi \; \pad^{m-i+1}\eta \; \pad^{m-1}\pay \eta \, dx.
    \end{align*}
    We claim that the above boundary term  vanishes. Let us separate the cases based on the parity of \(m\) and \(i\):
    \begin{itemize}
        \item If \(m\) is odd, then  \(\pad^{m-1}\pay \eta(.,\pm1)=0\) since $m-1<m$ is even and \(\eta\in H^m_e(\Omega)\cap H^{m+1}(\Omega)\). 
        \item Let \(m\) be even. 
        \begin{itemize}
            \item If \(2\le i\) is odd, then  \(\pad^{m-i+1}\eta (.,\pm1)=0\) since $m-i+1<m$ is even and \(\eta\in H^m_e(\Omega)\).
            \item If \(2\le i\) is even, then \(\pad^i \psi (.,\pm1) = \pad^{i-2}\p_2u_1(.,\pm1)=0\) since $i-2<m-1$ is even and $\p_2u_1\in H^{m}_e\subset H^{m-1}_e$.
        \end{itemize}
    \end{itemize}
   On the other hand, the first two terms in $J_i$ can be estimated using  H\"{o}lder's inequality and Sobolev embeddings as before, giving 
    \[
|J_i| \le C_m \hnorm{u}{m}^2\,\hnorm{\eta}{m}.
    \]
   We have proven that
\bq\label{est:I}
       |I| \le C_m (\hnorm{u}{m}^2\hnorm{\eta}{m} + \pnorm{\na u_2}{\infty} \hnorm{\eta}{m}^2).
\eq
   {\bf 2.}  As for $II$, we  integrate by parts in \(\pay\) and recall \(\pay \eta = \Delta \psi\) to obtain
    \begin{multline*}
        II=\int_\Omega D^m(\rho'_s\, \pay \psi)\, D^m \eta = -\int_\Omega D^m(\rho'_s\, \psi) \, D^m\pay \eta \nonumber  \\ 
         = -\int_\Omega D^m (\rho'_s\, \psi) \, D^m \Delta \psi \nonumber 
        = \int_\Omega \na D^m (\rho'_s\,\psi) \, \cdot \na D^m \psi \mp \int_{y=\pm1} D^m(\rho'_s\, \psi) \, \pad D^m \psi\, dx. \label{eq:termii}
\end{multline*}
    We claim that the boundary term vanishes. Expanding results in 
    \[
    \int_{y=\pm1} D^m(\rho'_s\, \psi) \, \pad D^m \psi\, dx
     = \sum_{0\le i\le m} a_{i,m} \int_{y=\pm1} D^i\rho'_s\; D^{m-i}\psi \; D^m \pad \psi \, dx.
    \]
    Let \(D^m=\p_1^a\p_2^b\) and \(D^i = \p_1^c\p_2^d\), where $c\le a$ and $d\le b$.
    \begin{itemize}
        \item If \(b\) is odd, then \( D^m\p_2\psi(.,\pm1)= \p_1^a\p_2^{b-1}\p_2u_1(.,\pm1) = 0\) since $b-1<m$ is even and $\p_2 u_1\in H^m_e$.
        \item Let \(b\) be even.
        \begin{itemize}
            \item If \(d\) is odd, then \(D^i \rho'_s(.,\pm1) =\p_1^c\p_2^{d+1}\rho_s=\p_1^c\p_2^{d-1}\rho''_s= 0\) since $d-1<m$ is even and  \(\rho_s'' \in H^{m}_e\) for $\rho_s\in \mathcal{S}_m$. 
            \item Let \(d\) be even. Then $D^{m-i}\psi(.,\pm1) =\p_1^{a-c}\p_2^{b-d}\psi(\cdot, \pm 1)$, where $b-d$ is even.
                        \begin{itemize}
                        \item If $a=c$ and $ b=d$, then $D^{m-i}\psi(.,\pm1)=\psi(\cdot, \pm 1)=0$ by Lemma \ref{lemm:stream}. 
                \item If \(a>c\), then \(D^{m-i}\psi(.,\pm1) =- \p_1^{a-c-1}\p_2^{b-d}u_2(\cdot, \pm 1)=0\) since $b-d<m+1$ is even and $u_2\in H^{m+1}_e$.
                \item If \(b>d\), then \(D^{m-i} \psi(.,\pm1) = \pay^a \pad^{b-d-2}\p_2u_1(.,\pm1) = 0\) since  $b-d-2<m-1$ is even and $\p_2u_1\in H^{m}_e\subset H^{m-1}_e$.
            \end{itemize}
        \end{itemize}
    \end{itemize}
It follows that
\bq\label{est:II1}
       II= \int_\Omega \rho_s'\, |\na D^m \psi|^2 + \int_\Omega [\na D^m, \rho'_s]\psi\cdot \na D^m \psi  \le -c_0\hnorm{u}{m}^2  + \|\rho'_s\|_{H^{m+1}}\hnorm{\psi}{m}\hnorm{\psi}{m+1},
\eq
    where we used the assumption \(\sup\rho'_s < -c_0<0\) and bounded the \(L^2\) norm of the commutator by \(\|\rho'_s\|_{H^{m+1}}\hnorm{\psi}{m}\). Since
\[
        \p_2 \int_\T \psi(x,y)\,dx = \int_\T u_1(x,y)\,dx = -\int_\T \p_1 p(x,y)\,dx = 0
\]
and  \(\psi(x,\pm1) =0\), we have $  \int_\T \psi(x,y)\,dx  =0$ for all $y$. Hence,  Poincar\'{e}'s inequality yields
    \bq\label{Poincare:psi}
    \tnorm{\psi}\le C\tnorm{\pay \psi} \le C \tnorm{u},\quad \| \psi\|_{H^{m+1}}\le C\| u\|_{H^m}.
    \eq
By interpolating $\| \psi\|_{H^m}$ between $\| \psi\|_{L^2}$ and $\|\psi\|_{H^{m+1}}$, we deduce from \eqref{est:II1} and \eqref{Poincare:psi} that
    \bq\label{est:II}
 II  \le -\frac{3}{4}c_0\hnorm{u}{m}^2  + C\|u\|_{L^2}^2,\quad C=C(\|\rho'_s\|_{H^{m+1}}, m). 
     \eq
   Finally, \eqref{estimate:thetahm} follows from \eqref{est:I} and \eqref{est:II}.
\end{proof}
Next, we prove estimates for $\| \p_1^2\eta\|^2_{L^2}$ and $\| \p_1\p_2\eta\|^2_{L^2}$ with dissipative terms $ \tnorm{\pay \na u_2}^2 $ and $\hnorm{u_2}{2}^2$, respectively. These will subsequently lead to the $t^{-m}$-time-average decay \eqref{bound:h2decay} of $\| u_2\|_{H^2}^2$, and then the time-integrability \eqref{est:intu2} of $\| \na u_2\|_{L^\infty}$. 
\begin{prop}\label{prop:h2difineq}
Let $2<m\in \Nn$ and $\rho_s\in \mathcal{S}_m$, with $\sup \rho_s'<-c_0<0$. Suppose that  $\eta\in C([0, T], H^m_e)$ is a solution of \eqref{eq:pipm}. There exist positive constants $c_1=c_1(m)$ and  $C_1=C_1(c_0, m)$ and \(C_2= C_2(\|\rho_s''\|_{W^{1, \infty}}, c_0, m)\) such that the following inequalities hold on $(0, T)$:
\bq\label{p11theta} 
            \ddt \tnorm{\pay^2\eta}^2 \le \Big(-\frac{7}{4}c_0+c_1\| \eta\|_{H^m}\Big) \tnorm{\pay \na u_2}^2 + C_1(\hnorm{u}{m}^2 + \pnorm{\p_2u_2}{\infty})\tnorm{\pay^2 \eta}^2 + C_2\tnorm{u_2}^2 
            \eq
            and
            \bq  \label{p12theta}
            \begin{aligned}
            \ddt \tnorm{\p_1\p_2 \eta}^2 &\le \Big(-\frac74c_0+c_1\| \eta\|_{H^m}\Big)\hnorm{u_2}{2}^2 + C_1(\hnorm{u}{m}^2+\pnorm{\p_2u_2}{\infty})\tnorm{\p_1\p_2\eta}^2 \\ 
           & \quad +c_0 \tnorm{\p_1\na u_2}^2 +  c_1\tnorm{\p_1\na u_2} \tnorm{\p_1\p_2\eta}\hnorm{\eta}{m} +C_2 \tnorm{u}^2.
            \end{aligned}
\eq
\end{prop}
\begin{proof}
   {\bf 1. } We take  \(\p_1^2\) of \eqref{eq:pmass} and multiply the resulting equation  by \(\p_1^2\eta\) to have
    \begin{align}
        \mez \ddt \tnorm{\p_1^2\eta}^2 = -\int_\Omega \p_1^2(u\cdot\na\eta)\, \p_1^2\eta - \int_\Omega  \rho_s'\p_1^2u_2\, \p_1^2\eta =: A + B.
    \end{align}
    Since \(u\) is incompressible, an integration by parts gives
    \begin{align*}
    A &= -\int_\Omega (\p_1^2u \cdot \na \eta + 2\p_1 u \cdot \na  \p_1\eta)\, \p_1^2\eta \\ 
    &= -\int_\Omega (\p_1^2u_1\, \p_1\eta + 2\p_1u_1\, \p_1^2\eta)\, \p_1^2\eta  -\int_\Omega (\p_1^2 u_2 \, \p_2\eta + 2\p_1u_2\, \p_2\p_1\eta) \, \p_1^2\eta  =: A_1 + A_2.
    \end{align*}
    Using Cauchy-Schwarz's inequality followed by Young's inequality, we bound 
    \begin{align*}
        |A_1| &\le \tnorm{\p_1^2u_1}\pnorm{\p_1\eta}{\infty}\tnorm{\p_1^2\eta} + 2 \pnorm{\p_1u_1}{\infty} \tnorm{\p_1^2\eta}^2 \\ 
        &\le \frac{\d}{2} \tnorm{\p_1^2u_1}^2 + C_\d \tnorm{\p_1^2\eta}^2\pnorm{\p_1\eta}{\infty}^2 + 2\pnorm{\p_1u_1}{\infty}\tnorm{\p_1^2\eta}^2.
    \end{align*}
Combining  \eqref{ineq:pthetahk} with the embedding $H^{m-1}(\Omega)\hookrightarrow L^\infty(\Omega)$ yields $\| \p_1\eta\|_{L^\infty}\le C\| u\|_{H^m}$. Consequently, invoking  the incompressibility of \(u\), we obtain 
    \begin{align} \label{bound:a111}
        |A_1| \le \frac{\d}{2}\tnorm{\p_1\p_2 u_2}^2 + C_\delta(\hnorm{u}{m}^2 + \pnorm{\p_2u_2}{\infty}) \tnorm{\p_1^2\eta}^2.
    \end{align}
    As for $A_2$, we integrate by parts in \(x\):
    \[
    A_2 = \int_\Omega \p_1u_2\, \p_2\eta \, \p_1^3\eta - \int _\Omega \p_1u_2 \, \p_2\p_1\eta \, \p_1^2\eta =:A_{21} + A_{22}.
    \]
    Since \(\p_1u_2=0\) on $\p\Omega$, we have  $\p_1u_2\, \p_2\eta\in H^1_0(\Omega)$, and hence
    \begin{align*}
    |A_{21}| &\le \| \p_1u_2\, \p_2\eta \|_{H^1_0} \|\p_1^3\eta\|_{H^{-1}} \le C \|\na(\p_1u_2\, \p_2\eta)\|_{L^2} \hnorm{\p_1^3 \eta}{-1}.
    \end{align*}
       Since \(m > 2,\) there exists $q>2$ and $p\in (2, \infty)$ such that \(H^m \hookrightarrow W^{2,q}\) and $\frac{1}{p}+\frac{1}{q}=\mez$. The embedding $H^1(\Omega)\hookrightarrow L^p(\Omega)$ holds and we have
\begin{align*}
    \|\na(\p_1u_2\, \p_2\eta)\|_{L^2} &\le \tnorm{\p_1\na u_2}\|\p_2 \eta\|_{L^\infty}+ \pnorm{\p_1u_2}{p}\|\na \p_2\eta\|_{L^q}  \\
    &\le \tnorm{\p_1\na u_2}\|\p_2 \eta\|_{L^\infty}+ C \|\p_1 u_2\|_{H^1} \hnorm{\eta}{m} \\
    &\le C \tnorm{\p_1\na u_2}{\hnorm{\eta}{m}},
\end{align*}
where we have used the Poincar\'e inequality $ \|\p_1 u_2\|_{H^1} \le C\| \na \p_1 u_2\|_{L^2}$.

    To bound \(\hnorm{\p_1^3\eta}{-1},\) we let \(\vp \in H^1_0\) and recall that \(\p_1\eta = \Delta \psi, \, u_2 = -\p_1\psi\). Then we can integrate by parts to have
    \begin{align*}
        \int_\Omega \p_1^3 \eta \, \vp &= \int_\Omega \p_1^2 \Delta\psi\, \vp 
        = -\int_\Omega \p_1^2 \na \psi \cdot \na \vp = \int_\Omega \p_1\na u_2\cdot \na \vp \le \tnorm{\p_1 \na u_2}\|\vp\|_{H^1_0}.
    \end{align*}
    Hence, \(\hnorm{\p_1^3\eta}{-1} \le \tnorm{\p_1\na u_2}\) and 
    \begin{align}\label{bound:a2111}
        |A_{21}| \le c_1\tnorm{\p_1\na u_2}^2 \hnorm{\eta}{m},\quad c_1=c_1(m).
    \end{align}
    By an analogous argument, we obtain
    \bq  \label{bound:a2211}\begin{aligned}
        A_{22} &\le \pnorm{\p_1u_2}{p}\pnorm{\p_2\p_1\eta}{q}\tnorm{\p_1^2\eta}
        \le C\tnorm{\p_1\na u_2}\pnorm{\p_2\Delta \psi}{q} \tnorm{\p_1^2\eta} \\ 
        &\le C \tnorm{\p_1\na u_2} \hnorm{u}{m} \tnorm{\p_1^2\eta}\le  \frac{\d}{2} \tnorm{\p_1\na u_2}^2 + C_\d \hnorm{u}{m}^2 \tnorm{\p_1^2\eta}^2.
    \end{aligned}
    \eq
    It follows from  \eqref{bound:a111}, \eqref{bound:a2111}, and \eqref{bound:a2211} that
    \bq\label{bound:Atotal}
    |A|\le   \tnorm{\p_1\na u_2}^2(\delta+ c_1\hnorm{\eta}{m})+C_\delta(\hnorm{u}{m}^2 + \pnorm{\p_2u_2}{\infty}) \tnorm{\p_1^2\eta}^2,\quad C(\delta)=C(\delta, m).
    \eq
Inserting  \(\p_1^2\eta=-\Delta u_2\) in $B$, we  can integrate by parts to have
    \begin{align*}
        B &=\int_\Omega  \rho_s'\p_1^2u_2\Delta u_2=-\int_\Omega \na(\rho_s'\p_1^2u_2)\cdot \na u_2+\int_{y\pm 1} \rho_s'\p_1^2u_2\p_2u_2 dx.
                    \end{align*}
   Since \(u_2(.,\pm1) = 0\),  The boundary term vanishes, and hence
    \begin{align*}
        B =-\int_\Omega \p_1\na(\rho_s'\p_1u_2)\cdot \na u_2 =\int_\Omega \na(\rho_s'\p_1u_2)\cdot \p_1\na u_2=\int_\Omega \rho_s'|\p_1\na u_2|^2+\int_\Omega \rho_s''\p_1u_2 \p_1\p_2 u_2.
    \end{align*}
    Using the upper bound $\rho'_s< -c_0<0$ and the Cauchy-Schwarz inequality, we get
    \begin{align*} 
    B &\le  -c_0 \tnorm{\p_1\na u_2}^2 + \pnorm{\rho''_s}{\infty} \tnorm{\p_1u_2}\tnorm{\p_2\p_1u_2} \\
    &\le -c_0 \tnorm{\p_1\na u_2}^2 + \frac{ \pnorm{\rho''_s}{\infty}^2}{\delta} \tnorm{\p_1u_2}^2 + \d \tnorm{\p_1\na u_2}^2 .
    \end{align*}
  On the other hand, we have
    \[
    \tnorm{\p_1 u_2}^2 = -\int_\Omega u_2\, \p_1^2u_2  \le \tnorm{u_2} \tnorm{\p_1\na u_2} 
    \le  \alpha  \tnorm{\p_1\na u_2}^2+\frac{1}{\alpha}  \tnorm{u}^2,\quad \alpha>0.
    \]
    Thus,
    \begin{align} \label{bound:b11}
        B &\le \Big(-c_0+\d +  \frac{\alpha\pnorm{\rho''_s}{\infty}^2}{\delta}\Big) \tnorm{\p_1 \na u_2}^2 + \frac{ \pnorm{\rho''_s}{\infty}^2}{\delta \alpha} \tnorm{u}^2.
    \end{align}
In view of  \eqref{bound:Atotal} and \eqref{bound:b11},  we arrive at \eqref{p11theta} by choosing $\delta= \frac{c_0}{32}$ and  $\alpha= \frac{\delta c_0}{16 \pnorm{\rho''_s}{\infty}^2}$.

    {\bf 2.} As for \eqref{p12theta}, we have 
    \begin{align*}
        \mez \ddt \tnorm{\p_2\p_1\eta}^2 =  - \int_\Omega \p_2\p_1 (u\cdot\na \eta )\, \p_2\p_1\eta -\int_\Omega \p_2\p_1(\rho_s'\, u_2)\, \p_2\p_1\eta=:A' + B'.
    \end{align*}
The incompressibility of \(u\) implies
    \begin{align*}
    A' &= -\int_\Omega  \Big(\p_2\p_1u_1\,\p_1\eta + \p_2u_1\, \p_1^2\eta + \p_1u_1\, \p_2\p_1\eta \Big) \, \p_2\p_1\eta -\int_\Omega  \Big(\p_2\p_1u_2\,\p_2\eta + \p_2u_2\, \p_1\p_2\eta + \p_1u_2\, \p_2^2\eta \Big) \, \p_2\p_1\eta \\
    &=:A'_1+A'_2.
    \end{align*}
Using the Sobolev embedding \(H^m\hookrightarrow W^{1,\infty} \),  incompressibility, and the relation \(\p_1 \eta = \Delta \psi \), we  obtain
    \begin{align*}
        |A'_1| &\le \Big(\tnorm{\p_2\p_1u_1}\pnorm{\p_1\eta}{\infty} + \pnorm{\p_2u_1}{\infty} \tnorm{\p_1^2\eta} + \pnorm{\p_1u_1}{\infty} \tnorm{\p_2\p_1\eta} \Big)\, \tnorm{\p_2\p_1\eta}\\
       & \le  \Big(\tnorm{\p_2^2u_2}\pnorm{\Delta \psi}{\infty} + \pnorm{\p_2u_1}{\infty} \tnorm{\Delta u_2} + \pnorm{\p_1u_1}{\infty} \tnorm{\p_2\p_1\eta} \Big)\, \tnorm{\p_2\p_1\eta}\\
       &\le C \Big(\hnorm{u_2}{2}\hnorm{u}{m}  + \pnorm{\p_2u_2}{\infty}\tnorm{\p_2\p_1\eta} \Big)\tnorm{\p_2\p_1\eta}.
    \end{align*}
It follows that
       \begin{align}\label{bound:c112}
        |A'_1|       &\le \d \hnorm{u_2}{2}^2  + C_\d (\hnorm{u}{m}^2 + \pnorm{\p_2 u_2}{\infty})\tnorm{\p_2\p_1\eta}^2,\quad C_\delta=C(\delta, m).
    \end{align}
As for \(A'_2\), we integrate by parts in its first term:
    \begin{align*}
    -\int_\Omega \p_2\p_1u_2\, \p_2\eta \, \p_2\p_1\eta  &= \int_\Omega \p_2u_2\, (\p_2\p_1\eta)^2  + \int_\Omega \p_2u_2\, \p_2\eta \, \p_2\p_1^2\eta \\ 
    &= \int_\Omega \p_2u_2\, (\p_2\p_1\eta)^2 - \int_\Omega \p_2^2u_2\, \p_2\eta \, \p_1^2\eta -\int_\Omega \p_2u_2\, \p_2^2\eta \, \p_1^2\eta \pm \int_{y=\pm1} \p_2u_2\, \p_2\eta \, \p^2_1\eta,
    \end{align*}
    where the boundary term vanishes since \(\eta(.,\pm1) = 0.\) Inserting this into the definition of $A'_2$ yields
    \begin{align*}
        A'_2 = -\int_\Omega \p_2^2u_2\, \p_2\eta \, \p_1^2\eta - \int_\Omega \p_2u_2\, \p_2^2\eta \, \p_1^2\eta -\int_\Omega \p_1u_2\, \p_2^2\eta \, \p_2\p_1\eta 
        =:A'_{21}+A'_{22} + A'_{23}.
    \end{align*}
    Since \(\p_1^2\eta = -\Delta u_2,\) we have
    \begin{align}\label{bound:c2112}
        |A'_{21}| \le \| \p_2^2u_2\|_{L^2}\pnorm{\p_2\eta}{\infty}\| \Delta u_2\|_{L^2} \le c_1 \hnorm{\eta}{m}\hnorm{u_2}{2}^2,\quad c_1=c_1(m).
    \end{align}
    With the same choice of \(p,q\) as in the estimates for $A_{21}$ above, we obtain
    \begin{align}\label{bound:c2212}
        |A'_{22}| \le \pnorm{\p_2 u_2}{p}\|\p_2^2\eta\|_{L^{q}}\|\Delta u_2\|_{L^2}\le c_1 \hnorm{\eta}{m}\hnorm{u_2}{2}^2
    \end{align}
    and 
    \begin{align}\label{bound:c2312}
        |A'_{23}| \le \pnorm{\p_1u_2}{p}\|\p_2^2\eta\|_{L^q}\tnorm{\p_2\p_1\eta}
        \le c_1 \tnorm{\p_1\na u_2}\hnorm{\eta}{m}\tnorm{\p_2\p_1\eta}.
    \end{align}
      Combining \eqref{bound:c112}, \eqref{bound:c2112}, \eqref{bound:c2212}, and \eqref{bound:c2312}, we find 
      \bq\label{bound:c1000}
      |A'|\le (\d+2c_1\|\eta\|_{H^m})\hnorm{u_2}{2}^2 +c_1 \tnorm{\p_1\na u_2}\hnorm{\eta}{m}\tnorm{\p_2\p_1\eta} + C_\d (\hnorm{u}{m}^2 + \pnorm{\p_2 u_2}{\infty})\tnorm{\p_2\p_1\eta}^2.
      \eq
    To bound $B'$, we first integrate by parts and use that $\p_1^2\eta=-\Delta u_2$:
    \begin{multline*}
   B'=-\int_\Omega \p_2\p_1(\rho_s'u_2)\p_2\p_1\eta=\int_\Omega \p_2(\rho_s'u_2)\p_2\p^2_1\eta=-\int_\Omega \p_2(\rho_s'u_2)\Delta \p_2u_2\\
   =\int_\Omega \p_2\na (\rho_s'u_2)\cdot \na \p_2u_2\mp \int_{y=\pm 1}\p_2(\rho_s'u_2)\p^2_2u_2dx,
    \end{multline*}
    where the boundary term vanishes since $u_2\in H^m_e$ by Proposition \ref{prop:hme}. It follows that 
    \begin{align*}
       B'&=\int_\Omega \rho_s'|\na \p_2u_2|^2+\int_\Omega \rho_s''\na u_2\cdot \na \p_2 u_2+\int_\Omega \rho_s''\p_2u_2\p_2^2u_2+\int_\Omega \rho_s'''u_2\p^2_2u_2\\
&\le -c_0\| \na \p_2u_2\|_{L^2}^2+ \| \rho_s''\|_{L^\infty}\| \na u_2\|_{L^2}\| \na \p_2u_2\|_{L^2}+\| \rho_s'''\|_{L^\infty}\|u_2\|_{L^2}\| \na \p_2u_2\|_{L^2}\\\
&= -c_0(\hnorm{u_2}{2}^2 - \tnorm{\p_1\na u_2}^2 - \tnorm{u_2}^2)+ \| \rho_s''\|_{L^\infty}\| \na u_2\|_{L^2}\| \na \p_2u_2\|_{L^2}+\| \rho_s'''\|_{L^\infty}\|u_2\|_{L^2}\| \na \p_2u_2\|_{L^2}.
       \end{align*}
By interpolating $H^1$ between $L^2$ and $H^2$, we have
       \[
        \| \rho_s''\|_{L^\infty}\| \na u_2\|_{L^2}\| \na \p_2u_2\|_{L^2}\le \delta \|u_2\|_{H^2}^2+C_\delta \| u\|_{L^2}^2,\quad C_\delta=C(\delta,  \| \rho_s''\|_{L^\infty}, m).
        \]
       Consequently, we obtain 
       \bq\label{bound:c1001}
        B'\le  (-c_0+2\delta)\hnorm{u_2}{2}^2+c_0 \tnorm{\p_1\na u_2}^2 +C_\delta \| u\|_{L^2}^2,\quad C_\delta=C(\delta,  \| \rho_s''\|_{W^{1, \infty}}, c_0, m).
        \eq 
      In view of \eqref{bound:c1000} and \eqref{bound:c1001}, we can choose $\delta=\frac{c_0}{24}$ to conclude  \eqref{p12theta}.
    \end{proof}
\subsection{Nonlinear stability}
\subsubsection{Time-average decay} 
We record the following facts about time-average decay from differential inequalities. These results are taken from \cite[ Section 2.2]{BJPW}.
\begin{lemm}\label{lemm:dineq1}
    Suppose that \(f(t)\) and $a(t)$ are nonnegative functions on \([0,T]\) such that 
    \[
    \ddt f(t) \le -a(t)^{-\a} f(t)^\b
    \]
    for some constants \(\a > 0, \, \b>1 \). Then we have 
    \[
    f(t) \le (\beta-1)^{-\frac{1}{\beta-1}} \frac{A^{\a/(\b-1)}}{t^{(\a+1)/(\b-1)}}\quad\forall t\in (0, T],~A:=\displaystyle{\int_0^ta(s)\,ds}.
    \]
\end{lemm}
\begin{lemm}\label{lemm:dineq2}
    Suppose that  $f(t)$ and $g(t)$ are nonnegative functions on \([0,T]\) such that 
    \[
    \ddt f(t) \le -g(t) \text{ and } f(t) \le \frac{A}{t^n}
    \]
    for some constants  $n>0$ and $A>0$. Then we have 
    \[
    \frac{2}{t}\int_{t/2}^t g(s)\, ds \le \frac{2^{n+1}A}{t^{n+1}}\quad\forall t\in (0, T]. 
    \]
\end{lemm}
\begin{lemm}\label{lemm:dineq3}
   Suppose that  \(f(t)\) is a nonnegative function on \([0,T]\) such that 
    \[
    \frac{2}{t} \int_{t/2}^t f(s)\, ds \le \frac{A}{t^n}
    \]
    for some constant $n>1$ and \(A > 0\). Then we have
    \[
    \int_t^T f(s)\, ds \le \frac{A}{2(1-2^{1-n})t^{n-1}}\quad\forall  t\in (0, T].
    \]
    
\end{lemm}

\begin{lemm}\label{lemm:dineq4}
    Suppose that  $H$, $h$, $f$, and $g$ are nonnegative functions on $[0, T]$ such that 
    \[
    \ddt H(t) \le -f(t) + h(t)H(t) + g(t).
    \]
    Assume that there exist $n$, $A>0$ such that 
    \[
    \int_0^Th(s)\,ds \le 1,\quad  \frac{2}{t} \int_{t/2}^t H(s)\,ds \le \frac{A}{t^n}, \quad  \frac{2}{t}\int_{t/2}^t g(s)\,ds \le \frac{A}{t^{n+1}}\quad\forall t\in (0, T]. 
    \]
    Then we have 
         \[
    \frac{2}{t} \int_{t/2}^t f(s)\,ds \le\frac{ 8e(2^{n-1}+1)A}{t^{n+1}}\quad\forall t\in (0, T].
    \]
\end{lemm}
\subsubsection{Bootstrap} In this subsection, we fix $T>0$ and assume  the following bootstrap assumption:  
\begin{align} 
\max_{t\in [0, T]}\hnorm{\rho(t) - \rho_s}{m}^2  + \int _0 ^T \hnorm{u}{m}^2\le   M\eps^2<\eps_0^2,\label{assum:boot}
\end{align}
where $\eps$ and $M$ will be chosen later, and $\eps_0=\eps_0(c_0,  m)$ is the small constant in \eqref{bound:hkperturb}. For $\eps_0$ sufficiently small, this implies 
\bq\label{monotone:rho}
\| \rho(t)-\rho_s\|_{W^{1, \infty}}< \frac{c_0}{2},\quad \sup_{x\in \Omega}\p_2 \rho(x, t)< -\frac{c_0}{2}  \quad\forall t\le T. 
\eq
\begin{prop}\label{prop:potentialalgdecay}
Given \eqref{assum:boot}, there exists $C_\sharp=C_\sharp(\| \rho'_s\|_{L^\infty}, c_0, m)>0$ such that
      \bq\label{tad:pE}
        \mathcal{E}\bigl(\rho(t)\bigr) \le \frac{C_\sharp M\eps^2}{t^m}\quad   \text{ and } \quad
        \frac{2}{t}\int_{t/2}^t \tnorm{u(s)}^2\, ds \le \frac{C_\sharp M\eps^2}{t^{m+1}}\quad\forall t\in (0, T].
    \eq
\end{prop}
\begin{proof}
  Using  \eqref{potentialup}, \eqref{monotone:rho},  and interpolation, we obtain 
    \[
    \mathcal{E}\bigl(\rho(t)\bigr) \le C\| u\|_{H^1}^2\le C(\tnorm{u}^{\frac{m-1}{m}}\hnorm{u}{m}^{\frac{1}{m}})^2,\quad C=C(\| \rho'_s\|_{L^\infty}, c_0, m).
    \]
 Since \eqref{assum:boot} implies that $\max_{t\in [0, T]} \| \rho(t)-\rho_s\|_{H^m}\le \eps_0$, we can apply Corollary \ref{prop:potentialdecay} to have
    \[
    \ddt \mathcal{E}\bigl(\rho(t)\bigr) = -\tnorm{u}^2 \le -C\mathcal{E}\bigl(\rho(t)\bigr)^{\frac{m}{m-1}}\hnorm{u(t)}{m}^{\frac{-2}{m-1}}.
    \]
    Then, Lemma \ref{lemm:dineq1} with $\alpha=\frac{1}{m-1}$ and $\beta=\frac{m}{m-1}$ gives us the first inequality in \eqref{tad:pE}. 
  The second inequality in \eqref{tad:pE} follows from the first one and  Lemma \ref{lemm:dineq2}.
\end{proof}
\begin{lemm}\label{lemm:nau2integ}
 Let   $T_1\in (0, T]$ and assume that
    \begin{equation}\label{bound:h2decay:00}
    \frac{2}{t} \int_{t/2}^t \hnorm{u_2}{2}^2 \le \frac{C_0M\eps^2}{t^m}\quad\forall t\le T_1.
    \end{equation}
 Then under \eqref{assum:boot}, there exists a constant  $K=K(C_0, m)>0$ independent of $T_1$ such that     
\bq\label{est:intu2}
    \int_{0}^{T_1}\pnorm{\na u_2}{\infty} \le K\sqrt{M}\eps.
    \eq
\end{lemm}
\begin{proof}
Let   \( \a \in (0,  \frac{m-2}{2})\). By interpolating  \(H^{2+\a}\) between \(H^2\) and \(H^m\), we obtain
    \[
    \pnorm{\na u_2}{\infty}\le K_1\| u_2\|_{H^{2+\a}} \le K_2 \hnorm{u_2}{2}^{\frac{m-2-\a}{m-2}}\, \hnorm{u_2}{m}^{\frac{\a}{m-2}},\quad K_2=K_2(m, \a).
    \]
    Using H\"{o}lder's inequality twice, we have
    \begin{align*}
    \frac{2}{t} \int_{t/2}^t \pnorm{\na u_2}{\infty}\, ds &\le \bigl( \frac{2}{t}\int_{t/2}^t \pnorm{\na u_2}{\infty}^2\, ds \bigr)^{\mez }  \le   K_2\bigl(\frac{2}{t} \int_{t/2}^t \hnorm{u_2}{2}^{\frac{2(m-2-\a)}{m-2}}\, \hnorm{u}{m}^{\frac{2\a}{m-2}}\bigr)^\mez \\
    &\le K_2 \bigl(\frac{2}{t} \int_{t/2}^t \hnorm{u_2}{2}^2\bigr)^{\frac{m-2-\a}{2(m-2)}}\, \bigl(\frac{2}{t}\int_{t/2}^t \hnorm{u}{m}^2\bigr)^{\frac{\a}{2(m-2)}},\quad t\le T_1.
    \end{align*}
    Then, combining the $H^2$ bound  \eqref{bound:h2decay:00} for $u_2$ with the $H^m$ bound \eqref{assum:boot} for $u$, we deduce 
    \begin{align*}
        \frac{2}{t} \int_{t/2}^t \pnorm{\na u_2}{\infty}\, ds & \le K_2\bigl(\frac{C_0M\eps^2}{t^m}\bigr)^{\frac{m-2-\a}{2(m-2)}}\, \bigl(\frac{2M\eps^2}{t}\bigr)^{\frac{\a}{2(m-2)}}
        \le K_3\sqrt{M}\eps \frac{1}{t^{m\frac{m-2-\a}{2(m-2)}+\frac{\a}{2(m-2)}}},\quad t\le T_1,
    \end{align*}
    where $ K_3=K_3(C_0, \a, m)$. Since $m>2$, we can choose $\a>0$ small enough so that $\delta:=m\frac{m-2-\a}{2(m-2)}+\frac{\a}{2(m-2)}-1>0$. Then,  Lemma \ref{lemm:dineq3} implies
    \[
    \int_t^{T_1} \pnorm{\na u_2}{\infty} \le K_4\frac{\sqrt{M}\eps}{t^\delta},\quad K_4=K_4(C_0, m).
    \]
   On the other hand,  \eqref{assum:boot} implies 
    \[
    \int_0^{t} \pnorm{\na u_2}{\infty} \le t^\mez \bigl(\int_0^t  \pnorm{\na u_2}{\infty}^2\bigr)^\frac12
    \le K_5 t^\mez\bigl(\int_0^t \hnorm{u}{m}^2\bigr)^\mez \le K_5t^\mez \sqrt{M}\eps,\quad K_5=K_5(m).
    \]
From the previous two estimates, we obtain \eqref{est:intu2} with $K=K(C_0, m)$.
    \end{proof}

\begin{prop}\label{prop:h2decay}
   There is a positive  constant $\mu=\mu(\eps_0, \|\rho_s''\|_{W^{1, \infty}}, c_0, m)\le 1$ such that  if  $M\eps^2<\mu$ and \eqref{assum:boot} holds, then there  exists a constant \(C_0=C_0(\|\rho_s''\|_{W^{1, \infty}}, c_0, m)>0\) independent of \((T, M, \eps)\) such that the following holds for all $t\in (0, T]$:
    \begin{equation}\label{bound:h2decay}
    \frac{2}{t} \int_{t/2}^t \hnorm{u_2}{2}^2 \le \frac{C_0M\eps^2}{t^m}.
    \end{equation}
\end{prop}
\begin{proof}
    First, we note that  \eqref{assum:boot} implies   \eqref{bound:h2decay} with strictly inequality for all \( t < \min\{1, T\}\), provided $C_0\ge 1$. For a large $C_0$ that will  be determined, we assume for the sake of contradiction that \eqref{bound:h2decay} does not  hold for all \(t \le T.\) Then there exists a minimum \(T^*<T\) such that 
    \begin{equation}
    \frac{2}{T^*} \int_{T^*/2}^{T^*}\hnorm{u_2}{2}^2 = \frac{C_0M\eps^2}{{T^*}^m}. \label{u2h2contra}
    \end{equation}
   The remainder of this proof consists of two steps:  1)  we apply Lemma \ref{lemm:dineq4} to the differential inequality \eqref{p11theta} to get a decay rate $\frac{1}{t^{m+1}}$ for \(\tnorm{\p_1\na u_2}^2\); 2) using this rate, we apply Lemma \ref{lemm:dineq4} to \eqref{p12theta} to reach a contradiction.

    Let $C_1$ and $C_2$ be the constants in \eqref{p11theta}. We set $C_3=\max\{C_1, C_2, c_0, c_1\}$, $H_1(t) = \tnorm{\p_1^2\eta}^2$, $f_1(t) = \frac32 c_0 \tnorm{\p_1\na u_2}^2$,  $h_1(t)=C_3(\hnorm{u}{m}^2+\pnorm{\p_2u_2}{\infty})$,  and $g_1(t) = C_3\tnorm{u_2}^2$. We impose that $c_1\sqrt{M}\eps\le \frac{1}{4}c_0$,  so that \eqref{assum:boot} implies $(-\frac{7}{4}c_0+c_1\| \eta\|_{H^m}) \tnorm{\pay \na u_2}^2\le -f_1$. Hence, it follows from \eqref{p11theta} that
    \[
    \ddt H_1(t) \le -f_1(t) + h_1(t)\, H_1(t) + g_1(t).
    \]
Since  \(\p_1^2\eta =- \Delta u_2\), \eqref{u2h2contra} implies 
    \[
    \frac{2}{t}\int_{t/2}^t H_1 \le \frac{C_0M\eps^2}{t^m}\quad\forall t\le T^*.
    \]
    Choosing $M\eps^2\le \eps_0^2$, we have by virtue of Lemma  \ref{prop:potentialalgdecay} that
    \[
    \frac{2}{t}\int_{t/2}^t g_1(s)\,ds  \le \frac{C_3C_\sharp M\eps^2}{t^{m+1}}.
    \]
In view of \eqref{u2h2contra}, we can apply Lemma \ref{lemm:nau2integ} with $T_1=T^*\le T$  to have
    \[
    \int_0^{T^*}h_1(t)\,dt = \int_0^{T^*} C_3(\hnorm{u}{m}^2+\pnorm{\p_2u_2}{\infty}) \le C_3M\eps^2+K\sqrt{M}\eps,\quad K=K(C_0, m).
    \]
    We impose that 
    \[
    M\eps^2<\min\{1, (C_3+K)^{-2}\},
    \]
    so that $\int_0^{T^*}h_1<1$ and Lemma \ref{lemm:dineq4} yields 
    \begin{equation}\label{decay:p1nabu2}
        \frac{2}{t}\int_{t/2}^t \tnorm{\p_1\na u_2}^2= \frac{2}{t}\int_{t/2}^t \frac{2}{3c_0}f_1(s)ds \le \frac{4}{3c_0}8e(2^{m-1}+1)\max\{C_0, C_3C_\sharp  \}\frac{M\eps^2}{t^{m+1}}\le C_5(1+C_0) \frac{M\eps^2}{t^{m+1}}.
    \end{equation}
  Next, we set $H_2(t) = \tnorm{\p_1\p_2\eta}^2$, $f_2(t)= \frac32 c_0\hnorm{u_2}{2}^2$,  $h_2(t) = C_3(\hnorm{u}{m}^2+ \pnorm{\p_2u_2}{\infty})\equiv h_1(t)$, and $g_2(t) = C_3\bigl(\tnorm{\p_1\na u_2} \tnorm{\p_1\p_2\eta}\hnorm{\eta}{m}  +\tnorm{\p_1\na u_2}^2+ \tnorm{u}^2\bigr)$.  Under the condition $c_1\sqrt{M}\eps\le \frac{1}{4}c_0$ previously imposed,  \eqref{p12theta} reads
    \[
    \ddt H_2(t) \le -f_2(t) + h_2(t)\, H_2(t) + g_2(t).
    \]
    Since \(\p_1\p_2\eta = \p_2\Delta \psi = \Delta u_1,\) interpolating \(H^2\) between \(L^2\) and \(H^m\) gives  
\[
    \frac{2}{t}\int_{t/2}^t  \tnorm{\p_1\p_2\eta}^2   \le 
    \frac{2}{t} \int_{t/2}^t \tnorm{u}^{\frac{2(m-2)}{m}} \, \hnorm{u}{m}^{\frac{4}{m}}
    \le  \bigl(\frac{2}{t}\int_{t/2}^t \tnorm{u}^2\bigr)^{\frac{m-2}{m}} \bigl(\frac{2}{t}\int_{t/2}^t \hnorm{u}{m}^2\bigr)^{\frac{2}{m}}.
\]
Then we  invoke \eqref{assum:boot} and the second estimate in \eqref{tad:pE} to deduce 
    \bq \label{uh2decay}
      \frac{2}{t}\int_{t/2}^t H_2(s)\,ds   \le {C_6M\eps^2}\, (\frac{1}{t^{m+1}})^{\frac{m-2}{m}}\, (\frac{1}{t})^{\frac{2}{m}}
    = \frac{C_6M\eps^2}{t^{m-1}}.
      \eq
     We recall that \(\int_0^{T*} h_1(s)\,ds  < 1.\) As for the first term in \(g_2(t)\),  we use \eqref{assum:boot}, \eqref{decay:p1nabu2}, and \eqref{uh2decay}: 
         \begin{align*}
        \frac{2}{t} \int_{t/2}^t \tnorm{\p_1\na u_2} \, \tnorm{\p_1\p_2\eta} \hnorm{\eta}{m}
        &\le (\sup_{t\in[0,T]} \hnorm{\eta}{m}) \bigl(\frac{2}{t} \int_{t/2}^t \tnorm{\p_1\na u_2}^2\bigr)^\mez \, \bigl(\frac{2}{t} \int_{t/2}^t \tnorm{\p_1\p_2\eta}^2\bigr)^\mez  \\ 
        &\le C_7(M\eps^2)^\mez \, ( \frac{(1+C_0)M\eps^2}{t^{m+1}})^\mez \, \bigl( \frac{M\eps^2}{t^{m-1}})^\mez\\
        & =\frac{C_7(1+C_0)^\mez(M\eps^2)^\mez M\eps^2}{t^m}.
    \end{align*}
    For the second term in \(g_2,\) we interpolate between  \eqref{decay:p1nabu2} and \eqref{assum:boot}:
    \begin{align*}
        \frac{2}{t}\int_{t/2}^t \tnorm{\p_1\na u_2}^2 & = \bigl( \frac{2}{t} \int_{t/2}^t \tnorm{\p_1\na u_2}^2\bigr)^\frac{m-1}{m} \, \bigl(\frac{2}{t} \int_{t/2}^t \tnorm{\p_1\na u_2}^2\bigr)^\frac{1}{m}  \\ 
        & \le \bigl( \frac{2}{t} \int_{t/2}^t \tnorm{\p_1\na u_2}^2\bigr)^\frac{m-1}{m} \, \bigl(\frac{2}{t} \int_0^{T^*} \|u_2\|_{H^m}\bigr)^\frac{1}{m} \\
         &\le C_8(1+C_0)^\frac{m-1}{m}  M \eps^2 (\frac{1}{t^{m+1}})^\frac{m-1}{m} \, (\frac{1}{t})^\frac{1}{m}\\
         &= C_8(1+C_0)^\frac{m-1}{m}  \frac{M \eps^2}{t^m}.
             \end{align*}
  Similarly, we can interpolate between the bounds for $u$ in \eqref{tad:pE} and  \eqref{assum:boot} to have
  \[
   \frac{2}{t}\int_{t/2}^t \| u\|_{L^2}^2 \le  C_9 \frac{M \eps^2}{t^m}.
   \]
   By imposing that $M\eps^2\le 1$, we obtain
   \bq\label{g2:est}
   \frac{2}{t}\int_{t/2}^t g_2\le C_3\left[C_7(1+C_0)^\mez+C_8(1+C_0)^\frac{m-1}{m}+C_9\right]  \frac{M \eps^2}{t^m}.
   \eq
  In view of the bounds \eqref{uh2decay} and  \eqref{g2:est}, applying Lemma \ref{lemm:dineq4} gives
    \[
    \frac{2}{t} \int_{t/2}^t \hnorm{u_2}{2}^2 \le  \frac{2}{3c_0}8e(2^{m-1}+1)\left[C_6+C_3C_7(1+C_0)^\mez+2C_3C_8(1+C_0)^\frac{m-1}{m}\right]  \frac{M \eps^2}{t^m},\quad t\le T^*.
    \]
    Evaluating this at $t=T^*$ and recalling \eqref{u2h2contra}, we arrive at  
       \bq\label{final:u2bound:proof}
       \begin{aligned}
    \frac{C_0M\eps^2}{{T^*}^m}& = \frac{2}{T^*} \int_{T^*/2}^{T^*} \hnorm{u_2}{2}^2\\
    & \le \frac{16e(2^{m-1}+1)}{3c_0}\left[C_6+C_3C_7(1+C_0)^\mez+C_3C_8(1+C_0)^\frac{m-1}{m}+C_3C_9\right] \frac{M \eps^2}{{T^*}^m}.
    \end{aligned}
    \eq
   We observe that the powers of $C_0$ on the right-hand side of \eqref{final:u2bound:proof} are strictly less than $1$. Therefore,  by choosing $C_0=C_0(c_0, C_3, C_6, C_7, C_8, C_9)=C_0(\|\rho_s''\|_{W^{1, \infty}}, c_0, m)$ sufficiently large, we reach a contradiction,  provided 
    \[
    M\eps^2<\min\left\{\eps_0^2, \frac{c_0}{4c_1}, 1, (C_3+K)^{-2} \right\}=:\mu(\eps_0, c_0, c_1, C_0, C_3, m)=\mu(\eps_0, \|\rho_s''\|_{W^{1, \infty}}, c_0, m).
    \]
\end{proof}


    \subsubsection{Proof of Theorem \ref{theo:stability}}
   Let $\eta_0\in H^m_e$, with $m\ge 3$ and $\| \eta_0\|_{H^m}\le \eps$.  By the local well-posedness in Proposition \ref{prop:lwp},  \eqref{eq:pipm} has a unique solution $\eta\in C([0, T^0), H^m_e)$, where $T^0\in (0, \infty]$ is the maximal existence time. 
   
    We claim that there are positive constants $M$  sufficiently large and $\eps$  sufficiently small, both independent of $T^0$, such that  \eqref{assum:boot} holds for all $T<T^0$. This claim would then imply
   \[
   \lim\sup_{t\to T^0}\| \eta(t)\|_{H^m}\le \sqrt{M}\eps,
   \]
   and hence $T^0=\infty$ in view of the blowup criterion \eqref{blowup}. In order to prove the above claim, we will invoke the $H^m$ estimate in Proposition \ref{prop:aprioriestimate}, which requires the additional regularity $\eta(t)\in H^{m+1}(\Omega)$. To justify this additional regularity, we approximate $\eta_0\in H^m_e$ by $\eta_0^{[k]}=\proj_k \eta_0\in H^\infty_e$, where  $\proj_k$ is the projection onto $\text{span}\{ \omega_{p, q}: |p|\le k,~q\le k\}$. Since $\rho_s'\in H^{m+1}$ and $\rho_s''\in H^m_e\subset H^{m-1}_e$ for $\rho_s\in \mathcal{S}_m$, we can apply Proposition \ref{prop:lwp} with $m+1$ in place of $m$ to obtain a unique solution $\eta_k\in C([0, T^0_k), H^{m+1}_e)$, where $T^0_k$ is the maximal time. Therefore, it suffices to prove the above claim for $\eta_k$ and let $k\to \infty$ to obtain the claim for $\eta$. In what follows, we drop the dependence on $k$ to alleviate notation. 
   
   For $M$ and $\eps$ that satisfy $M\eps^2<\mu_0\le 1$ and will be chosen later, we assume for the sake of contradiction that \eqref{assum:boot} ceases to hold after some time  $T^*\in (0, T^0)$. Then we have 
    \begin{align}\label{assum:bootcontra}
        \hnorm{\eta(T^*)}{m}^2 + \int_{0}^{T^*}\hnorm{u}{m}^2 \ = M\eps^2
    \end{align}
    by continuity.  By Proposition  \ref{prop:aprioriestimate}, $\eta$ satisfies 
    \bq\label{Hmest:100}
    \mez \ddt \hnorm{\eta}{m}^2 \le -\frac34c_0\hnorm{u}{m}^2 
        + \frac{A}{2}(\pnorm{\na u_2}{\infty} \hnorm{\eta}{m}^2 +  \hnorm{u}{m}^2\hnorm{\eta}{m}+\tnorm{u}^2),\quad t<T^0,
        \eq
        where $A=2\max\{A_1, A_2\}$. From \eqref{assum:boot}  and  the embedding $H^m(\Omega)\hookrightarrow W^{1, \infty}(\Omega)$, we have
\[
\forall t\le T^*,~\|\na u(t)\|_{L^\infty}\le A_3 \| u(t)\|_{H^m}\le A_4 \| \eta(t)\|_{H^m}\le A_4,\quad A_j=A_k(m).
\]
Inserting this into \eqref{Hmest:100}, we obtain $\ddt \hnorm{\eta}{m}^2 \le A_5 \hnorm{\eta}{m}^2$,  whence 
    \bq\label{crude:Hmeta}
     \hnorm{\eta{(t})}{m}^2 \le \| \eta(0)\|_{H^m}^2e^{A_5t}\le  \eps^2e^{A_5t},\quad t\le T^*.
     \eq
      We then insert this bound into \eqref{assum:bootcontra} to have
    \begin{align*}
    M\eps^2 = \hnorm{\eta{(T^*)}}{m}^2 + \int_0^{T^*} \hnorm{u}{m}^2
    \le \hnorm{\eta(T^*)}{m}^2 + A_6\int_0^{T^*}\hnorm{\eta}{m}^2
    \le A_7\eps^2e^{A_5T^*},\quad A_7=A_7(m).
    \end{align*}
     This implies $\log M\le \log A_7+A_5T^*$. We choose $M>A_7$ and set 
     \[
     \tilde{T}=\frac{1}{2A_5}(\log M-\log A_7)\in (0, \mez T^*). 
     \]
     
      Then, we integrate  \eqref{Hmest:100} between \(\tilde{T}\) and \(T^*>\tilde T\) to obtain
    \begin{align}\label{estimate:hminteg}
        \hnorm{\eta(T^*)}{m}^2 - \hnorm{\eta(\tilde{T})}{m}^2 + \frac32c_0 \int_{\tilde{T}}^{T^*} \hnorm{u}{m}^2 &\le 
        A\sup_{t \in [0,T^*]} \hnorm{\eta(t)}{m}^2 \int_{\tilde{T}}^{T^*}\pnorm{\na u_2}{\infty} \nonumber \\ 
        & \qquad+ A\sup_{t \in [0,T^*]}\| \eta(t)\|_{H^m}\int_{\tilde{T}}^{T^*}\hnorm{u}{m}^2 +A \int_{\tilde{T}}^{T^*} \tnorm{u}^2.
    \end{align}
    The exponential bound \eqref{crude:Hmeta} implies 
    \[
     \hnorm{\eta(\tilde{T})}{m}^2 \le \eps^2e^{A_5\tilde{T}}=\eps^2 A_7^{-\frac{1}{2}}M^\mez .
    \]
   For $M\eps^2<\mu$, we can apply  Proposition \ref{prop:h2decay},  Lemma \ref{lemm:nau2integ} , and  \eqref{assum:boot}  to  have 
    \[
  \sup_{t \in [0,T^*]} \hnorm{\eta(t)}{m}^2   \int_{\tilde{T}}^{T^*}\pnorm{\na u_2}{\infty} +   \sup_{t \in [0,T^*]}\| \eta(t)\|_{H^m}\int_{\tilde{T}}^{T^*}\hnorm{u}{m}^2\le (K+1)(M\eps^2)^\tdm.
    \]
    On the other hand, in view of the second decay estimate in \eqref{tad:pE}, an application of Lemma \ref{lemm:dineq3} yields  
    \[
     \int_{\tilde{T}}^{T^*} \tnorm{u}^2\le \frac{A_8 M\eps^2}{\tilde{T}^m}=\frac{A_9 M\eps^2}{(\log M-\log A_7)^m}.
    \]
    Inserting the above estimates into \eqref{estimate:hminteg}, we find 
    \[
     \hnorm{\eta(T^*)}{m}^2+ \frac32c_0 \int_{\tilde{T}}^{T^*} \hnorm{u}{m}^2 \le \eps^2 A_7^{-\frac{1}{2}}M^\mez+A(K+1)(M\eps^2)^\tdm+\frac{AA_9 M\eps^2}{(\log M-\log A_7)^m}.
     \]
    Combining this with equation \eqref{assum:bootcontra}, we obtain 
           \[
   \min\{1, \tdm c_0\}\le \frac{A_7^{-\frac{1}{2}}}{M^\mez}+A(K+1)(M\eps^2)^\mez+\frac{AA_9}{(\log M-\log A_7)^m}.
     \]
     This yields a contradiction if we choose $M=M(c_0, A, K, m)=M(c_0, \|\rho'_s\|_{H^{m+1}}, m)$ sufficiently large, followed by  $\eps<\eps_0=(\mu/M)^\mez$ sufficiently small. Therefore, \eqref{assum:boot} holds for all $T>0$. In view of \eqref{monotone:rho},  we can apply Propositions \ref{prop:potentialdecay} and \ref{prop:potentialalgdecay} to have $\rho^*(\cdot, t)=\rho^*(\cdot, 0)$ for all $t>0$, and in conjunction with \eqref{potentialequiv}, 
     \begin{align*}
\frac{c_0}{4\|\p_2\rho(x, t)\|_{L^\infty}^2} \| \rho(\cdot, t)- \rho^*(\cdot, 0)\|_{L^2}^2&= \frac{c_0}{4\|\p_2\rho(x, t)\|_{L^\infty}^2}\| \rho(\cdot, t)- \rho^*(\cdot, t)\|_{L^2}^2\\
&\le     \mathcal{E}(t) \le \frac{C_\sharp M\eps^2}{t^m},\quad t>0.
     \end{align*}
 Since $\|\p_2\rho(x, t)\|_{L^\infty}\le \frac{c_0}{2}+\| \rho_s'\|_{L^\infty}$, this implies the decay estimate \eqref{decay:mainthm}.

\appendix 
\section{Biot-Savart law}
We will consider $\Omega=\T\times (-1, 1)$ or $\Omega=\T\times \Rr$. First, we prove the existence of a stream function with good vanishing properties. 
\begin{lemm}\label{lemm:stream}
Suppose that $(u, p)\in H^1(\Omega)\times H^1(\Omega)$ satisfy Darcy's law. Then $u$ admits a stream function $\psi\in H^2(\Omega)$, i.e.,  $u= \nabla^\perp \psi := (\partial_2 \psi, -\partial_1 \psi)$, and $\psi\vert_{\partial\Omega} = 0$ for $\Omega=\T\times (-1, 1)$.
\end{lemm}
\begin{proof}
{\it Case 1:} $\Omega=\T\times (-1, 1)$. We define $\psi(x, y)=\int_{-1}^y u_1(x, y')dy'$, so that $\psi\in H^2(\Omega)$, $\p_y\psi =u_1$ and $\psi(\cdot, -1)=0$. Moreover, since $ \dv u=0$ and $u_2(x, \pm 1)=(u\cdot n)(x, \pm 1)=0$, we have 
\[
\p_x\psi (x, y)=-\int_{-1}^y \p_2u_2(x, y')dy'=-u_2(x, y)+u_2(x, -1)=-u_2(x, y). 
\]
In particular, we have $\p_x \psi(x, 1)=-u_2(x, 1)=0$, and hence $\psi(x, 1)=c$ is constant. On the other hand, the first component of Darcy's law gives 
\[
c=\psi(x, 1)=\int_{-1}^1 u_1(x, y')dy'=-\int_{-1}^1 \p_xp(x, y')dy'=-\p_x\int_{-1}^1 p(x, y')dy'.
\]
Since $p(x, z)$ is periodic in $x$, so is $\int_{-1}^1 p(x, z)dz$. This implies $c=0$. 

{\it Case 2:} $\Omega=\T\times \Rr$. We set $\psi(x, y)=-\int_{-\pi}^x u_2(x', y)dx'+c(y)$, so that $\p_1\psi =-u_2$. Since $\dv u=0$, we have $\p_2 \psi(x, y)=u_1(x, y)-u_1(-\pi, y)+c'(y)$. Choosing $c(y)=\int_0^y u_1(-\pi, y')dy'$, we get $\p_2 \psi=u_1$. Now $\psi$ is periodic in $x$ if and only if $d(y):=\int_{-\pi}^\pi u_2(x, y)dx=0$. The incompressibility  implies that $d$ is a constant. Moreover, since $|d|^2\le 2\pi\int_{-\pi}^\pi |u(x, y)|^2dx$ and $u\in L^2(\T\times \Rr)$, we deduce that $d=0$. We have proven that $\psi: \Omega\to \Rr$ satisfies  $\na^\perp \psi=u\in H^1(\Omega)$ and $\psi\in L^2_{loc}(\Omega)$. Next, we use Darcy's law to have 
\[
\frac{d}{dy}\int_\T \psi(x, y)dx=\int_\T u_1(x, y)dx=-\int_\T \p_xp(x,  y)dx=0.
\]  
Thus $\int_\T\psi(x, y)dx=m$ is constant. We redefine $\psi$ by $\psi -\frac{1}{|\T|}m$, so that $\na^\perp \psi=u$ and $\int_\T \psi(x, y)dx=0$. Consequently, $\psi\in L^2(\Omega)$ by Poincar\'e's inequality.
\end{proof}
\begin{lemm}\label{lemm:BS}
Suppose that $(\rho, u, p)\in H^1(\Omega)\times H^1(\Omega)\times H^1(\Omega)$ satisfy Darcy's law. Then the Biot-Savart law $ u= \nabla^\perp \Delta_D^{-1}\partial_1 \rho$ holds in $H^1(\Omega)$. 
\end{lemm}
\begin{proof}
 By Lemma \ref{lemm:stream}, there exists a stream function $\psi\in H^2(\Omega)$, which vanishes on the boundary if $\Omega=\T\times (-1, 1)$. Taking $\na^\perp$ of Darcy's law yields $\Delta \psi =\p_1\rho$. Thus, we obtain $\psi=\Delta_D^{-1}\p_1\rho$ and $ u= \nabla^\perp \Delta_D^{-1}\partial_1 \rho\in H^1(\Omega)$.
\end{proof}

\vspace{.1in}
{\noindent{\bf{Acknowledgment.}}   HQN was partially supported by the NSF CAREER Grant  DMS-2541807.



\end{document}